\documentclass[12pt,twoside]{amsart}
\usepackage{amssymb}
\usepackage{graphicx}
\usepackage[margin=1in]{geometry}

\newtheorem{theorem}{Theorem}[section]
\newtheorem{lemma}[theorem]{Lemma}

\newtheorem{conjecture}[theorem]{Conjecture}

\theoremstyle{definition}
\newtheorem{remark}[theorem]{Remark}

\usepackage{latexsym}
\usepackage{color,hyperref}
\definecolor{darkblue}{rgb}{0,0,0.6}
\hypersetup{colorlinks,breaklinks,linkcolor=darkblue,urlcolor=darkblue,anchorcolor=darkblue,citecolor=darkblue}

\title[Warnaar's bijection and colored partition identities, II]{Warnaar's bijection and colored partition identities, II}

\author{Colin Sandon}
\address{Department of Mathematics\\ MIT\\ Cambridge, MA 02139-4307}
\email{csandon@mit.edu}
\author{Fabrizio Zanello} \address{Department of Mathematics\\ MIT\\ Cambridge, MA 02139-4307\\{\tiny and}}
\address{Department of Mathematical  Sciences\\ Michigan Tech\\ Houghton, MI  49931-1295}
\email{zanello@math.mit.edu}
\thanks{2010 {\em Mathematics Subject Classification.} Primary: 05A17; Secondary: 05A19, 11P83, 05A15.\\
{\em Key words and phrases.} Partition identity; Colored partition; Farkas-Kra identity; Bijective proof; Warnaar's bijection.}

\begin{document}
\begin{abstract}
In our previous paper \cite{CSFZ1}, we determined a unified combinatorial framework to look at a large number of colored partition identities,  and studied the five identities corresponding to the exceptional modular equations of prime degree of the Schr\"oter, Russell and Ramanujan type. The goal of this paper is to use the master bijection of \cite{CSFZ1} to show combinatorially several new  and highly nontrivial colored partition identities. We conclude by listing a  number of further interesting identities of the same type as conjectures.
\end{abstract}

\maketitle

\section{Introduction}

This paper is the second of a series of two, begun with \cite{CSFZ1}, in which we study colored partition identities. Our project is motivated by the recent paper \cite{Be}, where B. C. Berndt determined and proved analytically the  colored partition identities corresponding to  five exceptional modular equations of prime degree that he defined ``of the Schr\"oter, Russell and Ramanujan type'', after the work of these three mathematicians (see \cite{BR,Ra,Ru1,Ru2,Sc}). We refer to the introduction to \cite{CSFZ1}, and of course to \cite{Be}, for more  details.

In \cite{CSFZ1}, responding to Berndt's call,  we determined a general and unified combinatorial framework in which to look at a number of colored partition identities, including the five of the Schr\"oter, Russell and Ramanujan type. In fact, extending S. Kim's idea from \cite{Ki}, in Theorem 2.3 in \cite{CSFZ1} we proved  that a   large family of colored partition identities is equivalent to  suitable equations in $(\nu_1,\dots,\nu_{t};d_1,\dots,d_{t})$, where the $\nu_i$ are partitions and the $d_i$ are integers whose sum is odd. This allowed us to show bijectively two more identities of the Schr\"oter, Russell and Ramanujan type (namely, those whose corresponding  modular equations have degrees 5 and 11). Thus, also thanks to the work of Kim \cite{Ki}, who  gave  the first bijective proofs of the identities modulo 3 and 7 (this latter also known as the ``Farkas-Kra identity''  \cite{FK}), now only the identity modulo 23 is open combinatorially. 

In this paper, we focus specifically on the case $t=12$ of the equivalent equations given by \cite{CSFZ1}, Theorem 2.3, and deduce  bijective proofs for a number of new, highly nontrivial colored partition identities. We believe that even more interesting identities of the same type hold, and we provide a large sample of these at the end, as conjectures.

\section{Preliminary results}

We begin by stating the main general result of \cite{CSFZ1}, which will be the key to bijectively show a number of  new and challenging partition identities in the next section. Its proof greatly generalized that of Kim \cite{Ki}, and used as a crucial ingredient  a bijection of S. O. Warnaar from \cite{Wa}. We state our theorem here in the particular case $t=12$ and $C_1=\dots=C_{12}=C$, which will suffice for our purposes. 

For  the main definitions of partition  theory, as well as three different introductions to this field, we refer the reader to  \cite{And,AE,Pak}.

\begin{theorem}\label{main}
Consider the equation
\begin{equation}\label{t}
C\sum_{i=1}^{12}|\mu_i|+C\sum_{i=1}^{12}{{d_i}\choose{2}} +\sum_{i=1}^{12}A_{i}d_{i}=C\sum_{i=1}^{12}|\alpha_i|+C\sum_{i=1}^{12}{{e_i}\choose{2}} +\sum_{i=1}^{12}B_{i}e_{i}+m,
\end{equation}
for given integers $C\geq 1$, $0\leq A_i\leq C/2$ and $0\leq B_i\leq C/2$ for all $i$, and $m\geq 0$. Let $S$ be the set containing one copy of all positive integers congruent to $\pm A_i$ modulo $C$ for each $i$, and $T$ the set containing one copy of all positive integers congruent to $\pm B_i$ modulo $C$ for each $i$. Let $D_S(N)$ (respectively, $D_T(N)$) be the number of partitions of $N$ into distinct elements of $S$ (respectively, $T$), where we require such partitions to have an odd number of parts if no $A_i$ (respectively, no $B_i$) is equal to zero. Set $$p=|\{B_i=0\}|-|\{A_i=0\}|,$$
adopting the convention that $|X|=1$ if $X=\emptyset$. Also, let  $P$ be the set of all partitions into positive integers.

Then the following are equivalent:

\begin{enumerate}
\item[(i)] For any $N\geq N_0\ge 1$, the number of tuples  $(\mu_1,\dots,\mu_{12};d_1,\dots,d_{12})$ such that the left-hand side of (\ref{t}) equals $N$, $\mu_i\in P$ and $d_i\in \mathbb Z$ for all $i$, and $\sum_{i=1}^{12}d_i$ is odd, is equal to the number of tuples $(\alpha_1,\dots,\alpha_{12};e_1,\dots,e_{12})$ such that the right-hand side of (\ref{t}) equals $N$,  $\alpha_i\in P$ and $e_i\in \mathbb Z$ for all $i$, and $\sum_{i=1}^{12}e_i$ is odd;

\item[(ii)] For any $N\geq N_0\ge 1$, $$D_S(N)=2^p\cdot D_T(N-m).$$
\end{enumerate}
\end{theorem}

\begin{proof}
See \cite{CSFZ1}, Theorem 2.3.
\end{proof}

The first of our preliminary lemmas was proved in  \cite{CSFZ1}. We recall its statement in the $t=12$, $C_1=\dots=C_{12}=C$ case for completeness.

\begin{lemma}\label{lemma1}
Fix arbitrary $C,A_1,\dots,A_{12},B_1,\dots,B_{12}$, such that $0\le A_i\le {C/2}$ and $0\le B_i\le C/2$, for all $i=1,\dots,12$. Let $S_N$ be the set of all tuples of $12$ partitions and $12$ integers $(\mu_1,\dots,\mu_{12};d_1,\dots,d_{12})$ such that $\sum_{i=1}^{12} d_i$ is odd and

\[C\sum_{i=1}^{12}|\mu_i|+C\sum_{i=1}^{12}{{d_i}\choose{2}}+\sum_{i=1}^{12} A_id_i =N.\]
Similarly, let $T_N$ be the set of all tuples of $12$ partitions and $12$ integers $(\alpha_1,\dots,\alpha_{12};e_1,\dots,e_{12})$ such that $\sum_{i=1}^{12} e_i$ is odd and

\[C\sum_{i=1}^{12}|\alpha_i|+C\sum_{i=1}^{12}{{e_i}\choose{2}} +\sum_{i=1}^{12} B_ie_i +m=N,\] 
where $m$ is an integer chosen so that the smallest value of $N$ for which $T_N\ne \emptyset$ is also the second smallest value of $N$ for which $S_N\ne \emptyset$. Define $k$ to be the smallest value such that $S_k\ne \emptyset$.  Further, let $U_N$ be the union of the set of all tuples of $12$ integers $(d_1,\dots,d_{12})$ such that $\sum_{i=1}^{12} d_i$ is odd and

\[C\sum_{i=1}^{12}{{d_i}\choose{2}}+\sum_{i=1}^{12} A_id_i =N,\]
with $|S_k|$ copies of the set of all tuples of $12$ integers $(f_1,\dots,f_{12})$ such that $\sum_{i=1}^{12} f_i$ is odd and 

\[C\sum_{i=1}^{12} \frac{f_i(3f_i-1)}{2}+k=N.\] 
Finally, let $V_N$ be the union of the  set of all tuples of $12$ integers $(e_1,\dots,e_{12})$ such that $\sum_{i=1}^{12} e_i$ is odd and

\[C\sum_{i=1}^{12}{{e_i}\choose{2}}+\sum_{i=1}^{12} B_ie_i+m =N,\]
with $|S_k|$ copies of the set of all tuples of $12$ integers $(f_1,\dots,f_{12})$ such that $\sum_{i=1}^{12} f_i$ is even and 

\[C\sum_{i=1}^{12} \frac{f_i(3f_i-1)}{2}+k=N.\] 
Then $|S_N|=|T_N|$ for all $N>k$ if and only if $|U_N|=|V_N|$ for all $N$.
\end{lemma}

\begin{proof}
See \cite{CSFZ1}, Lemma 3.8.
\end{proof}

\begin{lemma}\label{lemma2}
Fix arbitrary $C,A_1,\dots,A_{12},B_1,\dots,B_{12}$ and $A'_1,\dots,A'_{12},B'_1,\dots,B'_{12}$, such that $0\le A_i\le {C/2}$, $0\le B_i\le C/2$,  $0\le A'_i\le {C/2}$ and $0\le B'_i\le C/2$, for all $i$. Define $S_N$, $T_N$, $S_N'$, $T_N'$, $k$ and $k'$ as in Lemma \ref{lemma1}, and let $Q_N$, $R_N$, $Q_N'$ and $R_N'$ be, respectively,  the subsets of $S_N$, $T_N$, $S_N'$ and $T_N'$ in which the partitions are all equal to $\emptyset $. Then, if $|S_N|=|T_N|$ for all $N>k$, the following are equivalent:

\begin{enumerate}
\item[(i)] For all $N>k'$, $|S_N'|=|T_N'|.$

\item[(ii)] For all $N$, $$|S_k|\cdot|Q_{N+k'}'|+|S_{k'}'|\cdot|R_{N+k}|=|S_k|\cdot|R_{N+k'}'|+|S_{k'}'|\cdot|Q_{N+k}|.$$
\end{enumerate}
\end{lemma}

\begin{proof}
By Lemma \ref{lemma1}, $|S_N'|=|T_N'|$ for all $N>k'$ if and only if $|U_{N+k'}'|=|V_{N+k'}'|$ for all $N$. But by subtracting $|S_{k'}'|$ times $|U_{N+k}|=|V_{N+k}|$ from $|S_k|$ times $|U_{N+k'}'|=|V_{N+k'}'|$, and canceling out the elements of the form $(f_1, \dots,f_{12})$, this is easily seen to be equivalent to $$|S_k|\cdot|Q_{N+k'}'|-|S_{k'}'|\cdot|Q_{N+k}|=|S_k|\cdot|R_{N+k'}'|-|S_{k'}'|\cdot|R_{N+k}|$$
for all $N$, which is obviously equivalent to (ii).
\end{proof}

\begin{lemma}\label{lemma3}
Fix $C,A_1,\dots,A_{12},B_1,\dots,B_{12}$ such that $A_i+B_{13-i}=C/2$ for all $i=1,\dots,12$, and set $m=\sum_{i=1}^{12} A_i/2-3C/2$. Then, for any integers $e_1, \dots, e_{12}$, we have:
$$C\sum_{i=1}^{12}{{\frac{1}{2}-e_{13-i}}\choose{2}} +\sum_{i=1}^{12}A_i\left(\frac{1}{2}-e_{13-i}\right)=C\sum_{i=1}^{12}{{e_i}\choose{2}} +\sum_{i=1}^{12}B_ie_i+m.$$
\end{lemma}

\begin{proof} This can easily be verified algebraically.
\end{proof}

Notice that the previous lemma implies that we can, in a sense, view $d$-tuples and $e$-tuples as both being in the same set, namely
$$D=\left\{d\in \mathbb{Z}^{12}\cup \left(\mathbb{Z}+\frac{1}{2}\right)^{12}:\sum_{i=1}^{12} d_i\in2\mathbb{Z}+1\right\}.$$

Every tuple $(d_1,\dots,d_{12})\in D$ has a value of $C\sum_{i=1}^{12}{{d_i}\choose{2}} +\sum_{i=1}^{12}A_id_i$, and will in some sense be considered ``of negative type'' if the $d_i$ are half-integers, since it will come from the opposite side of the bijection as the tuples in which the $d_i$ are integers.

Finally, the last preliminary lemma is the following:

\begin{lemma}\label{lemma4}
Fix integers $C^\ast,A^\ast_1, \dots,A^\ast_4,B^\ast_1,\dots,B^\ast_4$ and $m^\ast$, such that $A^\ast_i+B^\ast_{5-i}=C^\ast/2$ and $0\le A^\ast_i\le {C^\ast/2}$ for each $i$, $A^\ast_1+A^\ast_4=A^\ast_2+A^\ast_3=C^\ast/2+m^\ast$, and the second-smallest possible values of the left-hand sides of the versions of equation (1) of Theorem \ref{main} with the coefficients below are equal to the smallest possible values of their respective right-hand sides. 

Then condition (i) of Theorem \ref{main} holds for $N_0=\min{B^\ast_i}+3m^\ast$,  $C=C^\ast$, $m=3m^\ast$, and
$$(A_1,\dots,A_{12})=(A^\ast_1,A^\ast_2,A^\ast_3,A^\ast_4,A^\ast_1,A^\ast_2,A^\ast_3,A^\ast_4,A^\ast_1,A^\ast_2,A^\ast_3,A^\ast_4),$$ $$(B_1,\dots,B_{12})=(B^\ast_1,B^\ast_2,B^\ast_3,B^\ast_4,B^\ast_1,B^\ast_2,B^\ast_3,B^\ast_4,B^\ast_1,B^\ast_2,B^\ast_3,B^\ast_4),$$
if and only if condition (i) of Theorem \ref{main} holds for $N_0=\min{A^\ast_i}$, $C'=C^\ast$, $m=m^\ast$, and
$$(A_1',\dots,A_{12}')=(B^\ast_1,B^\ast_2,B^\ast_3,B^\ast_4,A^\ast_1,A^\ast_2,A^\ast_3,A^\ast_4,A^\ast_1,A^\ast_2,A^\ast_3,A^\ast_4),$$ $$(B_1',\dots,B_{12}')=(B^\ast_1,B^\ast_2,B^\ast_3,B^\ast_4,B^\ast_1,B^\ast_2,B^\ast_3,B^\ast_4,A^\ast_1,A^\ast_2,A^\ast_3,A^\ast_4).$$

\end{lemma}

\begin{proof}
Using the terminology introduced in Lemma \ref{lemma1}, for the first equation we have $k=\min(A_1,\dots,A_4)$, and for the second equation, $$k'=\min(B_1,\dots,B_4)=C/2-\max(A_1,\dots,A_4)=C/2-(C/2+m-\min(A_1,\dots,A_4))=k-m.$$

Unless $\min(B_1,\dots,B_4)=0$, the left-hand side of each equation takes on its smallest value only when a variable with the smallest coefficient is $1$ and the rest are $0$. If two or fewer of the $B_i$ are 
$0$ this still holds. Also, it is clearly impossible for exactly three of them to be $0$. Finally, if they are all $0$, then the $A_i$  all equal $C/2$ and it is easy to check that the two left-hand sides take on their smallest values for $2\cdot 12=24$ and $2^3=8$ values of the variables, respectively. 

In all of these cases, there are three times as many ways for the left-hand side of the first expression to take on its smallest value as there are for the left-hand side of the second one to. 

Thus, it easily follows from Lemma \ref{lemma2} that proving the statement is tantamount to proving that condition (ii) of Lemma \ref{lemma2} holds for the values of $C$, $A_i$, $B_i$, $A'_i$, and $B'_i$ given above. Since $|S_k|=3|S'_{k'}|$, this is equivalent to the statement that $3|Q_{N+k'}'|+|R_{N+k}|=3|R_{N+k'}'|+|Q_{N+k}|$, so it suffices to show that for each $(d_1,\dots,d_{12})\in Q_{N+k}$ or $(e_1',\dots,e_{12}')\in 3R_{N+k'}'$ there is a corresponding $(e_1,\dots,e_{12})\in R_{N+k}$ or $(d_1',\dots,d_{12}')\in 3Q_{N+k'}'$ and vice-versa.

By Lemma \ref{lemma3}, we can view all $d$-tuples and $e$-tuples as being in the set $D=\{d\in \mathbb{Z}^{12}\cup(\mathbb{Z}+1/2)^{12}:\sum_{i=1}^{12} d_i\in2\mathbb{Z}+1\}$ and all $d'$-tuples and $e'$-tuples as being in the set $D'=\{d'\in \mathbb{Z}^{12}\cup(\mathbb{Z}+1/2)^{12}:\sum_{i=1}^{12} d_i'\in2\mathbb{Z}+1\}.$ Note that while $D$ and $D'$ contain the same tuples, they have different value functions.

Furthermore, for arbitrary integers $d_1', \dots, d_4'$, we have the identity
\[C\sum_{i=1}^{4} {{d_i}'\choose{2}}+\sum_{i=1}^4 B_i d_i'+m=C\sum_{i=1}^{4} {{\frac{1}{2}-d_{5-i}'}\choose{2}}+\sum_{i=1}^4 A_i \left(\frac{1}{2}-d_{5-i}'\right).\]

Now, consider a map $\phi$ which sends an element $(d_1',\dots,d_{12}')\in D'$ to a tuple $(d_1,\dots,d_{12})$ as follows:

\[d_i=\left(\frac{1}{2}-d'_{5-i}\right) \text{ for } 0< i\le 4; \phantom{x} d_i=d'_i \text{ otherwise}.\]

The image of the tuple $(d_1',\dots,d_{12}')$ has a value of
$$C\sum_{i=1}^{12}{{d_i}\choose{2}} +A_1(d_{1}+d_{5}+d_{9})+A_2(d_{2}+d_{6}+d_{10})+A_3(d_{3}+d_{7}+d_{11})+A_4(d_{4}+d_{8}+d_{12}),$$
just like the elements of $D$. Therefore, if we apply $\phi$ to one copy of $D'$, and $\phi$, combined with the map

$$d_i\rightarrow d_{i+4\pmod{12}} {\ }{\ }{\ } \text{or} {\ }{\ }{\ }  d_i\rightarrow d_{i-4\pmod{12}},$$
to the other two  copies of $D'$, then the union of $D$ with the  three copies of $D'$ can be bijectively mapped into the following set:

$$W=\left\{d\in (\mathbb{Z}^{4}\cup(\mathbb{Z}+1/2)^{4})\times (\mathbb{Z}^{4}\cup(\mathbb{Z}+1/2)^{4})\times (\mathbb{Z}^{4}\cup(\mathbb{Z}+1/2)^{4}):\sum_{i=1}^{12} d_i\in 2\mathbb{Z}+1\right\}.$$

For any  $(d_1,\dots,d_{12})\in W$, this element has a value of
$$C\sum_{i=1}^{12}{{d_i}\choose{2}} +A_1(d_{1}+d_{5}+d_{9})+A_2(d_{2}+d_{6}+d_{10})+A_3(d_{3}+d_{7}+d_{11})+A_4(d_{4}+d_{8}+d_{12}),$$
and it belongs on the left-hand side of the desired bijection if and only if the number of $d_i$ that are half-integers is 0 or 8. 

For any such element, let $x=d_1+d_2+d_3+d_4$, $y=d_5+d_6+d_7+d_8$, and $z=d_9+d_{10}+d_{11}+d_{12}$. Then the map

\small
\[d^\ast_i=\left(\frac{x}{2}-d_{5-i}\right) \text{ for } 0< i\le 4; \phantom{x} d^\ast_i=\left(\frac{y}{2}-d_{13-i}\right) \text{ for } 4< i\le 8; \phantom{x} d^\ast_i=\left(\frac{z}{2}-d_{21-i}\right) \text{ for } 8< i\le 12\]
\normalsize

will always send an element of $W$ to an element of $W$ with the same value. Clearly, it is also an involution. Furthermore, $x+y+z=\sum_{i=1}^{12} d_i$ is odd, so either one or three quadruples of elements are being changed from an integer to a half-integer or vice-versa. Therefore, this map always converts an element to an element of the opposite type, and is the desired bijection.
\end{proof}

\section{The new colored partition identities}

The goal of the rest of the paper is to show bijectively a number of new interesting partition identities, thanks to their equivalent formulation provided by Theorem \ref{main}. Like the two identities of  the Schr\"oter, Russell and Ramanujan type that we proved in \cite{CSFZ1}, most of these identities will turn out to have highly nontrivial proofs.

\begin{lemma}\label{1}
Condition (i) of Theorem \ref{main} holds for $N_0= 3$,  $C=2$, $m=3$, and
$$(A_1,\dots,A_{12})=(1,\dots,1), (B_1,\dots,B_{12})=(0,\dots,0).$$
\end{lemma}

\begin{proof}
We proceed in a similar way to the proof of \cite{CSFZ1}, Lemma 3.10 (the equation equivalent to the partition identity  of the Schr\"oter, Russell and Ramanujan type corresponding to the modular equation of degree 5). By Lemma \ref{lemma1}, one can easily check that the statement is equivalent to the existence of a bijection between the set of tuples $(d_1,\dots,d_{12})$ and $24$ copies of every tuple with odd sum $(f_1,\dots,f_{12})$, and the set of tuples  $(e_1,\dots,e_{12})$ and $24$ copies of every tuple with even sum $(f_1',\dots,f_{12}')$, such that, for every corresponding pair,  

\begin{align*}
&2\sum_{i=1}^{12}{{d_i}\choose{2}} +\sum_{i=1}^{12} d_i \text{{\ }{\ }{\ } or{\ }{\ } {\ }} 2\sum_{i=1}^{12} \frac{f_i(3f_i-1)}{2}+1=\\
&2\sum_{i=1}^{12}{{e_i}\choose{2}} +\sum_{i=1}^{12}0e_i+3\text{{\ }{\ }{\ } or{\ } {\ }{\ }} 2\sum_{i=1}^{12} \frac{f_i'(3f_i'-1)}{2}+1.
\end{align*}

By Lemma \ref{lemma3}, we can consider the $d$-tuples and $e$-tuples as both being in the set $D=\{d\in \mathbb{Z}^{12}\cup(\mathbb{Z}+1/2)^{12}:\sum_{i=1}^{12} d_i\in 2\mathbb{Z}+1\}$. An arbitrary element of $D$, $(d_1,\dots,d_{12})$, has a value of $2\sum_{i=1}^{12}{{d_i}\choose{2}} +\sum_{i=1}^{12} d_i=\sum_{i=1}^{12} d_i^2.$

Now, let

\small
\begin{align*}
V_{1}&=(-1, -1, -1, -1, -1, -1, -1, -1, -1, -1, -1, -1), V_{2}=(1, 1, 1, 1, 1, 1, -1, -1, -1, -1, -1, -1),\\
V_{3}&=(1, 1, 1, -1, -1, -1, 1, 1, 1, -1, -1, -1), \phantom{xxxxxxx}V_{4}=(1, -1, -1, 1, 1, -1, 1, 1, -1, 1, -1, -1),\\
V_{5}&=(-1, 1, -1, 1, -1, 1, 1, -1, 1, 1, -1, -1), \phantom{xxxxxxx}V_{6}=(-1, -1, 1, -1, 1, 1, -1, 1, 1, 1, -1, -1),\\
V_{7}&=(-1, -1, 1, 1, -1, 1, 1, 1, -1, -1, 1, -1), \phantom{xxxxxxx}V_{8}=(1, -1, -1, -1, 1, 1, 1, -1, 1, -1, 1, -1),\\
V_{9}&=(-1, 1, -1, 1, 1, -1, -1, 1, 1, -1, 1, -1), \phantom{xxxxxxx}V_{10}=(-1, 1, 1, -1, 1, -1, 1, -1, -1, 1, 1, -1),\\
V_{11}&=(1, 1, -1, -1, -1, 1, -1, 1, -1, 1, 1, -1), \phantom{xxxxxxx}V_{12}=(1, -1, 1, 1, -1, -1, -1, -1, 1, 1, 1, -1).
\end{align*}
\normalsize

Notice that these vectors are pairwise orthogonal. Also, for arbitrary $d\in D$ and $1\le i\le 12$, $d\cdot V_i$ is an odd integer. Let
\[r_i(d)=d-\frac{d\cdot V_i}{6}V_i.\]

Notice that $\|r_i(d)\|=\|d\|$, for any $i$ and $d$. If $d\cdot V_i\equiv 0 \pmod{3}$, then $\frac{d\cdot V_i}{6}$ is a half-integer, so $r_i(d)$ is an element of $D$ that corresponds to an $e$-tuple if $d$ corresponds to a $d$-tuple, and vice-versa. So, we can map every point in $D$ that has a dot product with any of the $V_i$ that is divisible by $3$ to a point of the opposite type and the same value by sending it to $r_i(d)$, where $i$ is the smallest integer such that $d\cdot V_i\equiv 0\pmod{3}$. 

Note that $r_i(d)\cdot V_j=d\cdot V_j$ for all $j\ne i$ because of the orthogonality of the vectors, and $r_i(d)\cdot V_i=-d\cdot V_i$. It is easy to check that $r_i(r_i(d))=d$. Therefore, this map is an involution. 

That just leaves the points in $D$ whose dot products with $V_i$ are not divisible by $3$ for any $i$. Let $d\in D$ be any such point. For each $i$, let $x_i$ be the nearest integer to $\frac{d\cdot V_i}{6}$, $y_i=d\cdot V_i-6x_i$, and $z=d-\sum_{i=1}^{12} \frac{x_i}{2}V_i$. For any $i$, $d\cdot V_i\equiv \pm 1\pmod{6}$, so $y_i=\pm 1$. 

By the Pythagorean Theorem, we have

\[\|d\|^2=\sum_{i=1}^{12} \frac{(d\cdot V_i)^2}{12}=\sum_{i=1}^{12} \frac{(6x_i+y_i)^2}{12}=\sum_{i=1}^{12} x_i(3x_i+y_i)+1.\]

Now, $z$ must be either a tuple of integers or a tuple of half-integers, and $z\cdot V_i=y_i=\pm1$ for each $i$. It is easy to check that the only tuples that fit these criteria are the $24$ in which one element equals $\pm1$ and the rest are $0$. Therefore, we can choose a bijection between the $24$ possible values of $z$ and the $24$ copies of each tuple $(f_1,\dots,f_{12})$, and then map $d$ to the copy of $(-x_1 y_1,\dots,-x_{12} y_{12})$ corresponding to $z$. It follows that
\[2\sum_{i=1}^{12} \frac{f_i(3f_i-1)}{2}+1=\|d\|^2.\]

Also, the $y_i$ are determined by $z$, and for any given choice of $z$, the only $d$ that maps to a given tuple $(f_1,\dots,f_{12})$ is $z-\sum_{i=1}^{12} \frac{y_i\cdot f_i}{2} V_i$. Furthermore, the entries of this $d$ are all half-integers if $\sum_{i=1}^{12} f_i$ is odd and integers if it is even. So, this map always takes elements of $D$ corresponding to tuples of $d$'s to tuples of $f$'s with an even sum, and elements of $D$ corresponding to tuples of $e$'s to tuples of $f$'s with an odd sum, as desired.
\end{proof}

\begin{theorem}\label{111}
Let $S$ be the set containing  24 copies of the odd positive integers, and $T$ the set containing 24 copies of the even positive integers. Then, for any $N\geq 3$,
$$D_S(N)=2048D_T(N-3).$$\end{theorem}

\begin{proof}
Straightforward from Theorem \ref{main} and Lemma \ref{1}.
\end{proof}

\begin{lemma}\label{2}
Condition (i) of Theorem \ref{main} holds for $N_0= 1$,  $C=2$, $m=1$, and
$$(A_1,\dots,A_{12})=(0,0,0,0,1,1,1,1,1,1,1,1), (B_1,\dots,B_{12})=(0,0,0,0,0,0,0,0,1,1,1,1).$$
\end{lemma}

\begin{proof}
Straightforward from Lemmas \ref{1} and  \ref{lemma4}.
\end{proof}

\begin{theorem}\label{222}
Let $S$ be the set containing $8$ copies of the even positive  integers and $16$ copies of the odd positive  integers, and $T$ the set containing $16$ copies of the even positive  integers and $8$ copies of the odd positive  integers. Then, for any $N\geq 1$,
$$D_S(N)=16D_T(N-1).$$\end{theorem}

\begin{proof}
Straightforward from Theorem \ref{main} and Lemma \ref{2}.
\end{proof}

\begin{lemma}\label{3}
Condition (i) of Theorem \ref{main} holds for $N_0=4$,  $C=6$, $m=3$, and
$$(A_1,\dots,A_{12})=(2,\dots,2), (B_1,\dots,B_{12})=(1,\dots,1).$$
\end{lemma}

\begin{proof}
By Lemmas \ref{1} and \ref{lemma2}, this statement is equivalent to condition (ii) of Lemma \ref{lemma2} holding for $$(A'_1,\dots,A'_{12})=(3,\dots,3), (B'_1,\dots,B'_{12})=(0,\dots,0).$$ One can see that $m'=9$, and $|S'_{k'}|=2|S_k|=24$. So, this is equivalent to the statement that $|Q_{N+k'}'|+2|R_{N+k}|=|R_{N+k'}'|+2|Q_{N+k}|$. Thus, it suffices to show that for each $(d_1,\dots,d_{12})\in 2Q_{N+k}$ or $(e_1',\dots,e_{12}')\in R_{N+k'}'$ there is a corresponding $(e_1,\dots,e_{12})\in 2R_{N+k}$ or $(d_1',\dots,d_{12}')\in Q_{N+k'}'$ and vice-versa.

By Lemma \ref{lemma3}, we can consider all $d$-tuples and $e$-tuples as being in the set  $D=\{d\in \mathbb{Z}^{12}\cup(\mathbb{Z}+1/2)^{12}:\sum_{i=1}^{12} d_i\in2\mathbb{Z}+1\}$ and all $d'$-tuples and $e'$-tuples as being in the set  $D'=\{d\in \mathbb{Z}^{12}\cup(\mathbb{Z}+1/2)^{12}:\sum_{i=1}^{12} d_i\in2\mathbb{Z}+1\}$. Note that $D$ and $D'$ are not really the same because they have different value functions. For arbitrary $d\in D$, there are three cases.

If $\sum_{i=1}^{12} d_i\equiv 5\pmod{6}$, we map $d$ to another element of $D$ with an equal value but the opposite type, using the map $d_i^\ast=d_i-(\sum_{j=1}^{12} d_j-2)/6$, for each $i$. This map always results in $d^\ast$ such that $\sum_{i=1}^{12} d_i^\ast-2=-(\sum_{i=1}^{12} d_i-2)\equiv 3\pmod{6},$ and it is an involution, so it cancels out all such $d$.

If $\sum_{i=1}^{12} d_i\equiv 3\pmod{6}$, we map $d$ to another element of $D$ with an equal value but the opposite type, using the map $d_i^\ast=\sum_{j=1}^{12} d_j/6-d_i$, for each $i$. This map always results in $d^\ast$ such that $\sum_{i=1}^{12} d_i^\ast=\sum_{i=1}^{12} d_i\equiv 3\pmod{6}$, and it is an involution, so it cancels out all such $d$.

Finally, if $\sum_{i=1}^{12} d_i\equiv 1\pmod{6}$, we map both copies of $d$ to elements of $D'$ with the same value and type, using the maps $d_i'=d_i-(\sum_{j=1}^{12} d_j-1)/6$ and $d_i'=-d_i+(\sum_{j=1}^{12} d_j-1)/6$, for all $i$. These maps are injective, and always result in $d'$ such that $\sum_{i=1}^{12} d_i'\equiv 1\pmod{6}$ and $\sum_{i=1}^{12} d_i'\equiv -1\pmod{6}$, respectively. Therefore, they are bijections from the subset of $D$ for which $\sum_{i=1}^{12} d_i\equiv 1\pmod{6}$ to the subsets of $D'$ for which $\sum_{i=1}^{12} d_i'\equiv \pm 1\pmod{6}$.

That just leaves the subset of $D'$ for which $\sum_{i=1}^{12} d_i'\equiv 3\pmod{6}$. We map any $d'$ in this subset to another element of the subset with the same value but the opposite type, using the map $d_i''=d_i'-(\sum_{j=1}^{12} d_j')/6$. This map is an involution, thus it cancels out all $d'$ in this subset. This completes the bijection and the proof of the lemma.
\end{proof}

\begin{theorem}\label{333}
Let $S$ be the set containing 12 copies of the even positive integers that are not multiples of 3, and $T$ the set containing 12 copies of the odd positive integers that are not multiples of 3. Then, for any $N\geq 4$,
$$D_S(N)=D_T(N-3).$$\end{theorem}

\begin{proof}
Straightforward from Theorem \ref{main} and Lemma \ref{3}.
\end{proof}

\begin{lemma}\label{4}
Condition (i) of Theorem \ref{main} holds for $N_0=3$,  $C=6$,  $m=3$, and
$$(A_1,\dots,A_{12})=(1,1,1,1,1,1,3,3,3,3,3,3), (B_1,\dots,B_{12})=(0,0,0,0,0,0,2,2,2,2,2,2).$$
\end{lemma}

\begin{proof}
By Lemmas \ref{3} and \ref{lemma2}, this statement is equivalent to condition (ii) of Lemma \ref{lemma2} holding for $$(A'_1,\dots,A'_{12})=(2,\dots,2), (B'_1,\dots,B'_{12})=(1,\dots,1).$$ One can see that $m'=3$ and $|S'_{k'}|=2|S_k|=12$, so this is equivalent to the statement that $ |Q_{N+k'}'|+2|R_{N+k}|= |R_{N+k'}'|+2|Q_{N+k}|$. Thus, it suffices to show that for each $(d_1,\dots,d_{12})\in 2Q_{N+k}$ or $(e_1',\dots,e_{12}')\in R_{N+k'}'$ there is a corresponding $(e_1,\dots,e_{12})\in 2R_{N+k}$ or $(d_1',\dots,d_{12}')\in Q_{N+k'}'$ and vice-versa.”

By Lemma \ref{lemma3}, we can consider all $d$-tuples and $e$-tuples as being in the set  $D=\{d\in \mathbb{Z}^{12}\cup(\mathbb{Z}+1/2)^{12}:\sum_{i=1}^{12} d_i\in2\mathbb{Z}+1\}$, and all $d'$-tuples and $e'$-tuples as being in the set  $D'=\{d\in \mathbb{Z}^{12}\cup(\mathbb{Z}+1/2)^{12}:\sum_{i=1}^{12} d_i\in2\mathbb{Z}+1\}$. For arbitrary $d\in D$, let $$x=d_1+d_2+d_3+d_4+d_5+d_6-d_7-d_8-d_9-d_{10}-d_{11}-d_{12}.$$

If $x\equiv 5\pmod{6}$, we map $d$ to another element of $D$ with an equal value but the opposite type, using the map
\[d_i^\ast=\left(d_i-\frac{x-2}{6}\right) \text{ for } 0<i\le 6;\phantom{x}d_i^\ast=\left(d_i+\frac{x-2}{6}\right) \text{ for } 6<i\le 12.\]

This map always results in $d^\ast$ such that $x^\ast-2=-(x-2)\equiv 3\pmod{6}$, and it is an involution, so it cancels out all such $d$.

If $x\equiv 3\pmod{6}$, we map $d$ to another element of $D$ with an equal value but the opposite type, using the map

\[d_i^\ast=\left(d_{i+6}+\frac{x}{6}\right) \text{ for } 0<i\le 6;\phantom{x}d_i^\ast=\left(d_{i-6}-\frac{x}{6}\right) \text{ for } 6<i\le 12.\]

This map always results in $d^\ast$ such that $x^\ast=x\equiv 3\pmod{6}$, and it is an involution, so it cancels out all such $d$.

Finally, if $x\equiv 1\pmod{6}$, we map both copies of $d$ to elements of $D'$ with the same value and type, using the two maps:

\[d'_i=\left(d_i-\frac{x-1}{6}\right) \text{ for } 0<i\le 6;\phantom{x}d'_i=\left(d_i+\frac{x-1}{6}\right) \text{ for } 6<i\le 12,\]
and

\[d'_i=\left(d_{i+6}+\frac{x-1}{6}\right) \text{ for } 0<i\le 6;\phantom{x}d'_i=\left(d_{i-6}-\frac{x-1}{6}\right) \text{ for } 6<i\le 12.\]

These maps are  involutions (provided one modifies the second one by replacing $x-1$ with $x+1$  every time it occurs), and always result in $d'$ such that $x'=x-2(x-1)\equiv 1\pmod{6}$ and $x'=-x+2(x-1)\equiv -1\pmod{6}$, respectively. Hence, they are bijections from the subset of $D$ for which $x\equiv 1\pmod{6}$ to the subsets of $D'$ for which $x\equiv \pm 1\pmod{6}$.

That just leaves the subset of $D'$ for which $x\equiv 3\pmod{6}$. We map any $d'$ in this subset to another element of the subset with the same value but the opposite type, using the map

\[d''_i=\left(d_i'-\frac{x}{6}\right) \text{ for } 0<i\le 6;\phantom{x}d''_i=\left(d_i'+\frac{x}{6}\right) \text{ for } 6<i\le 12.\]

This map is an involution, so it cancels out all $d'$ in this subset. This completes the bijection.
\end{proof}

\begin{theorem}\label{444}
Let $S$ be the set containing  6 copies of the odd positive integers and 6 more copies of the odd positive  multiples of 3, and $T$ the set containing  6 copies of the even positive integers and 6 more copies of the positive multiples of 6. Then, for any $N\geq 3$,
$$D_S(N)=32D_T(N-3).$$\end{theorem}

\begin{proof}
Straightforward from Theorem \ref{main} and Lemma \ref{4}.
\end{proof}

\begin{lemma}\label{5}
Condition (i) of Theorem \ref{main} holds for $N_0=2$,  $C=4$, $m=2$, and
$$(A_1,\dots,A_{12})=(1,1,1,1,1,1,1,1,2,2,2,2), (B_1,\dots,B_{12})=(0,0,0,0,1,1,1,1,1,1,1,1).$$
\end{lemma}

\begin{proof}
By Lemmas \ref{2} and \ref{lemma2}, and a renumbering of the variables, this statement is equivalent to condition (ii) of \ref{lemma2} holding for $C=4$, $$(A_1,\dots,A_{12})=(2,2,2,2,1,1,1,1,1,1,1,1), (B_1,\dots,B_{12})=(0,0,0,0,1,1,1,1,1,1,1,1),$$ $$(A'_1,\dots,A'_{12})=(2,2,2,2,0,0,0,0,2,2,2,2), (B'_1,\dots,B'_{12})=(0,0,0,0,0,0,0,0,2,2,2,2).$$ Note that $m'=2$, and $|S'_{k'}|=|S_k|=8$. So, this is equivalent to the statement that $|Q_{N+k'}'|+|R_{N+k}|=|R_{N+k'}'|+|Q_{N+k}|$. Thus, it suffices to show that for each $(d_1,\dots,d_{12})\in Q_{N+k}$ or $(e_1',\dots,e_{12}')\in R_{N+k'}'$ there is a corresponding $(e_1,\dots,e_{12})\in R_{N+k}$ or $(d_1',\dots,d_{12}')\in Q_{N+k'}'$ and vice-versa.

It is easy to check that the following map is a value-preserving bijection between $d$-tuples such that $x=d_1+d_2+d_3+d_4$ is odd, and $e$-tuples such that $e_1+e_2+e_3+e_4$ is odd:

\[e_i=\left(d_i-\frac{x-1}{2}\right) \text{ for } 0<i\le 4; \phantom{x}e_i=d_i \text{ otherwise}.\]

Furthermore, the same map is also a value-preserving bijection between $d'$-tuples such that  $x=d_1'+d_2'+d_3'+d_4'$ is odd, and $e'$-tuples such that $e_1'+e_2'+e_3'+e_4'$ is odd.

Given any $d$-tuple such that $d_1+d_2+d_3+d_4$ is not odd,
$$y=d_5+d_6+d_7+d_8-d_9-d_{10}-d_{11}-d_{12}\equiv d_5+d_6+d_7+d_8+d_9+d_{10}+d_{11}+d_{12}\pmod{2}$$
must be odd. If $y\equiv 1\pmod{4}$, we can map this tuple to a $d'$-tuple with the same value, using the map 
\small
\[d_i'=d_i \text{ for } 0<i\le 4;\phantom{x}d_i'=\left(d_i-\frac{y-1}{4}\right) \text{ for } 5<i\le 8;\phantom{x}d_i'=\left(d_i+\frac{y-1}{4}\right) \text{ for } 9<i\le 12.\]
\normalsize

This map is a value-preserving bijection between the set of $d$-tuples for which $y\equiv 1\pmod{4}$ and the set of $d'$-tuples for which $d_5'+d_6'+d_7'+d_8'-d_9'-d_{10}'-d_{11}'-d_{12}'\equiv 1\pmod{4}$. 

If $y\equiv 3\pmod{4}$, we can map this tuple to a $d'$-tuple with the same value, using the map 
\Small
\[d_i'=d_i \text{ for } 0<i\le 4;\phantom{x}d_i'=1-\left(d_{i+4}+\frac{y+1}{4}\right) \text{ for } 5<i\le 8;\phantom{x}d_i'=\left(-d_{i-4}+\frac{y+1}{4}\right) \text{ for } 9<i\le 12.\]
\normalsize

This map is a value-preserving bijection between the set of $d$-tuples for which $y\equiv 3\pmod{4}$ and the set of $d'$-tuples for which $d_5'+d_6'+d_7'+d_8'-d_9'-d_{10}'-d_{11}'-d_{12}'\equiv 3\pmod{4}$. So, together they form a value-preserving bijection between the set of $d$-tuples for which $d_5+d_6+d_7+d_8+d_9+d_{10}+d_{11}+d_{12}$ is odd and the set of $d'$-tuples for which $d_5'+d_6'+d_7'+d_8'+d_9'+d_{10}'+d_{11}'+d_{12}'$ is odd.

Furthermore, the exact same pair of maps forms a value-preserving bijection between the set of all $e$-tuples for which $e_5+e_6+e_7+e_8+e_9+e_{10}+e_{11}+e_{12}$ is odd and the set of $e'$-tuples for which $e_5'+e_6'+e_7'+e_8'+e_9'+e_{10}'+e_{11}'+e_{12}'$ is odd.

These partial bijections combine to give the desired bijection.
\end{proof}

\begin{theorem}\label{555}
Let $S$ be the set containing 8 copies of the positive integers that are not multiples of 4, and $T$ the set containing 8 copies of the positive integers that are not congruent to 2 modulo 4. Then, for any $N\geq 2$,
$$D_S(N)=8D_T(N-2).$$\end{theorem}

\begin{proof}
Straightforward from Theorem \ref{main} and Lemma \ref{5}.
\end{proof}

\begin{lemma}\label{6}
Condition (i) of Theorem \ref{main} holds for $N_0=1$,  $C=4$, $m=1$, and
$$(A_1,\dots,A_{12})=(0,0,0,1,1,1,1,2,2,2,2,2), (B_1,\dots,B_{12})=(0,0,0,0,0,1,1,1,1,2,2,2).$$
\end{lemma}

\begin{proof}

By Lemmas \ref{2} and \ref{lemma2}, this statement is equivalent to condition (ii) of Lemma \ref{lemma2} holding for $$(A'_1,\dots,A'_{12})=(0,0,0,0,2,2,2,2,2,2,2,2), (B'_1,\dots,B'_{12})=(0,0,0,0,0,0,0,0,2,2,2,2).$$ Note that $m'=2$, and $|S'_{k'}|=2|S_k|=8$. So, this is equivalent to the statement that $|Q_{N+k'}'|+2|R_{N+k}|= |R_{N+k'}'|+2|Q_{N+k}|$. Thus, it suffices to show that for each $(d_1,\dots,d_{12})\in 2Q_{N+k}$ or $(e_1',\dots,e_{12}')\in R_{N+k'}'$ there is a corresponding $(e_1,\dots,e_{12})\in 2R_{N+k}$ or $(d_1',\dots,d_{12}')\in Q_{N+k'}'$ and vice-versa.

Note that $d_1^{\ast}= 1-d_1$, $d_1''= 1-d_1'$, $e_1^{\ast}= 1-e_1$, and $e_1''= 1-e_1'$ are value-preserving bijections from the sets of all $d$-, $d'$-, $e$-, and $e'$-tuples with an odd sum, to the sets of all $d$-, $d'$-, $e$-, and $e'$-tuples with an even sum, respectively. So, the requirement that the tuples have an odd sum is irrelevant here and we can ignore it. 

Also, the coefficients of the first $3$ elements of each type of tuple are the same, and the coefficients of the last $3$ elements of each type of tuple are also the same. Thus, we can extend to the desired bijection any value-preserving bijection between the set containing $2$ copies of each tuple $(d_4,\dots,d_9)$ and a copy of each tuple $(e_4',\dots,e_9')$, and the set containing $2$ copies of each tuple $(e_4,\dots,e_9)$ and a copy of each tuple $(d_4',\dots,d_9')$, by having all maps leave the first three and last three elements of all tuples unchanged.

Now, let $X=\{(x_1,x_2,x_3,x_4,x_5,x_6)\in \mathbb{Z}^6\}$, and for each $x\in X$, assign $x$ a value of
\[4\sum_{i=1}^6 {x_i \choose 2}+0x_1+0x_2+0x_3+2x_4+2x_5+2x_6+1.\]

The map
$$x_1=\frac{\pm (d_4+d_5-d_6-d_7)+1}{2}, \phantom{x}x_2=\frac{d_4-d_5+d_6-d_7+1}{2}, \phantom{x}x_3=\frac{d_4-d_5-d_6+d_7+1}{2},$$
$$x_4=\frac{d_4+d_5+d_6+d_7-1}{2}, \phantom{x}x_5=d_8, \phantom{x}x_6=d_9$$
is a value-preserving bijection between the set of all copies of tuples $(d_4,\dots,d_9)$ for which $d_4+d_5+d_6+d_7$ is odd, and $X$. Similarly, the map
$$x_1=e_4, \phantom{x}x_2=e_5, \phantom{x}x_3=\frac{1\pm (e_6+e_7+e_8+e_9-1)}{2},$$
$$x_4=\frac{e_6+e_7-e_8-e_9}{2}, \phantom{x}x_5=\frac{e_6-e_7+e_8-e_9}{2}, \phantom{x}x_6=\frac{e_6-e_7-e_8+e_9}{2}$$
is a value-preserving bijection between the set of all copies of tuples $(e_4,\dots,e_9)$ for which $e_6+e_7+e_8+e_9$ is even, and $X$. Obviously, combining the two yields a value-preserving bijection between the set of all copies of tuples $(d_4,\dots,d_9)$ for which $d_4+d_5+d_6+d_7$ is odd and the set of all copies of tuples $(e_4,\dots,e_9)$ for which $e_6+e_7+e_8+e_9$ is even.

Also, the map
$$d_4'=\frac{1\pm(d_4+d_5+d_6+d_7-1)}{2},\phantom{x}d_5'=\frac{d_4+d_5-d_6-d_7}{2},\phantom{x}d_6'=\frac{d_4-d_5+d_6-d_7}{2},$$
$$d_7'=\frac{d_4-d_5-d_6+d_7}{2},\phantom{x}d_8'=d_8,\phantom{x}d_9'=d_9$$
is a value-preserving bijection between the set of all copies of tuples $(d_4,\dots,d_9)$ for which $d_4+d_5+d_6+d_7$ is even and the set of all tuples $(d_4',\dots,d_9')$. Similarly, the map
$$e_4'=e_4,\phantom{x}e_5'=e_5,\phantom{x}e_6'=\frac{1\pm(e_6+e_7-e_8-e_9)}{2},$$
$$e_7'=\frac{e_6-e_7+e_8-e_9+1}{2},\phantom{x}e_8'=\frac{e_6-e_7-e_8+e_9+1}{2},\phantom{x}e_9'=\frac{e_6+e_7+e_8+e_9-1}{2}$$
is a value-preserving bijection between the set of all copies of tuples $(e_4,\dots,e_9)$ for which $e_6+e_7+e_8+e_9$ is odd and the set of tuples $(e_4',\dots,e_9')$.

Combining these bijections and then extending them to $\mathbb{Z}^{12}$ yields the desired bijection.
\end{proof}

\begin{theorem}\label{666}
Let $S$ be the set containing 4 copies of the odd positive integers, 6 copies of the positive  multiples of 4, and 10 copies of the positive  integers that are congruent to 2 modulo 4; let $T$ the set containing 4 copies of the odd positive integers, 10 copies of the positive  multiples of 4, and 6  copies of the positive  integers that are congruent to 2 modulo 4. Then, for any $N\geq 1$,
$$D_S(N)=4D_T(N-1).$$\end{theorem}

\begin{proof}
Straightforward from Theorem \ref{main} and Lemma \ref{6}.
\end{proof}

\begin{lemma}\label{7}
Condition (i) of Theorem \ref{main} holds for $N_0=1$,  $C=4$, $m=1$, and
$$(A_1,\dots,A_{12})=(0,0,1,1,1,1,1,1,1,1,2,2), (B_1,\dots,B_{12})=(0,0,0,0,0,1,1,1,1,2,2,2).$$
\end{lemma}

\begin{proof}
By Lemmas \ref{6} and  \ref{lemma2}, this statement is equivalent to condition (ii) of Lemma \ref{lemma2} holding for $$(A'_1,\dots,A'_{12})=(0,0,0,1,1,1,1,2,2,2,2,2), (B'_1,\dots,B'_{12})=(0,0,0,0,0,1,1,1,1,2,2,2).$$ Note that $m'=1$, and $|S'_{k'}|=2|S_k|=4$. So, this is equivalent to the statement that $|Q_{N+k'}'|+2|R_{N+k}|=|R_{N+k'}'|+2|Q_{N+k}|$. Thus, it suffices to show that for each $(d_1,\dots,d_{12})\in 2Q_{N+k}$ or $(e_1',\dots,e_{12}')\in R_{N+k'}'$ there is a corresponding $(e_1,\dots,e_{12})\in 2R_{N+k}$ or $(d_1',\dots,d_{12}')\in Q_{N+k'}'$ and vice-versa.

Obviously, the identity map is a value-preserving bijection between the set of $e'$-tuples and the set containing one copy of each $e$-tuple. The maps $d_1^{\ast}= 1-d_1$, $d_1''= 1-d_1'$, and $e_1^{\ast}= 1-e_1$ are value-preserving bijections from the sets of all $d$-, $d'$-, and $e$-tuples with an odd sum, to the sets of all $d$-, $d'$-, and $e$-tuples with an even sum, respectively. So, the requirement that the tuples have an odd sum is irrelevant  and we can ignore it.

The following  is a value-preserving bijection between the set of all copies of $d$-tuples such that $d_3+d_4+d_5+d_{10}$ is odd, and the set containing the other copy of each $e$-tuple:

\[
 e_i=
  \begin{cases}
    &\frac{1\pm(d_3+d_4-d_5-d_{10})}{2} \phantom{x}\text{ for }i=3,\\
    &\frac{d_3-d_4+d_5-d_{10}+1}{2} \phantom{x}\text{ for }i=4,\\
    &\frac{d_3-d_4-d_5+d_{10}+1}{2} \phantom{x}\text{ for }i=5,\\
    &\frac{d_3+d_4+d_5+d_{10}-1}{2} \phantom{x}\text{ for }i=10,\\
    &d_i \phantom{x}\text{ otherwise}.
  \end{cases}
\]

Similarly, the following map is a value-preserving bijection between the set of all copies of $d$-tuples such that $d_3+d_4+d_5+d_{10}$ is even and the set of all $d'$-tuples:

\[
 d'_i=
  \begin{cases}
    &\frac{1\pm(d_3+d_4+d_5+d_{10}-1)}{2} \phantom{x}\text{ for }i=3,\\
    &\frac{d_3+d_4-d_5-d_{10}}{2} \phantom{x}\text{ for }i=8,\\
    &\frac{d_3-d_4+d_5-d_{10}}{2} \phantom{x}\text{ for }i=9,\\
    &\frac{d_3-d_4-d_5+d_{10}}{2} \phantom{x}\text{ for }i=10,\\
    &d_{i+2} \phantom{x}\text{ for }4\le i\le 7,\\
    &d_i \phantom{x}\text{ otherwise}.
  \end{cases}
\]

Combining these maps yields the desired bijection.
\end{proof}

\begin{theorem}\label{777}
Let $S$ be the set containing 8 copies of the odd positive integers and 4 copies of the even positive  integers, and $T$ the set containing 4 copies of the odd positive integers, 10 copies of the positive  multiples of 4, and 6  copies of the positive  integers that are congruent to 2 modulo 4. Then, for any $N\geq 1$,
$$D_S(N)=8D_T(N-1).$$
\end{theorem}

\begin{proof}
Straightforward from Theorem \ref{main} and Lemma \ref{7}.
\end{proof}



Combining the last two  partition identities, we immediately have:

\begin{theorem}\label{7'''}
Let $S$ be the set containing 8 copies of the odd positive integers and 4 copies of the even positive  integers, and $T$ the set containing 4 copies of the odd positive integers, 6 copies of the positive  multiples of 4, and 10 copies of the positive  integers that are congruent to 2 modulo 4. Then, for any $N\geq 1$,
$$D_S(N)=2D_T(N).$$\end{theorem}

\begin{proof}
Straightforward from Theorems \ref{666} and \ref{777}.
\end{proof}

\begin{lemma}\label{8}
Condition (i) of Theorem \ref{main} holds for $N_0=1$,  $C=6$, $m=1$, and
$$(A_1,\dots,A_{12})=(0,0,1,1,1,1,2,2,3,3,3,3), (B_1,\dots,B_{12})=(0,0,0,0,1,1,2,2,2,2,3,3).$$
\end{lemma}

\begin{proof}
Straightforward from Lemmas \ref{4} and  \ref{lemma4}.
\end{proof}

\begin{theorem}\label{888}
Let $S$ be the set containing 4 copies of the positive integers that are congruent to either 0 or $\pm 1$ modulo 6, 2 copies of the even positive integers that are not multiples of 3, and 8 copies of the odd positive multiples of 3; let $T$ be the set containing 4 copies of the positive integers that are congruent to either 3 or $\pm 2$ modulo 6, 2 copies of the  positive integers that are congruent to $\pm 1$ modulo 6, and 8 copies of the  positive multiples of 6. Then, for any $N\geq 1$,
$$D_S(N)=4D_T(N-1).$$\end{theorem}

\begin{proof}
Straightforward from Theorem \ref{main} and Lemma \ref{8}.
\end{proof}

\begin{lemma}\label{9}
Condition (i) of Theorem \ref{main} holds for $N_0=2$,  $C=6$, $m=1$, and
$$(A_1,\dots,A_{12})=(1,1,1,1,2,2,2,2,2,2,2,2), (B_1,\dots,B_{12})=(1,1,1,1,1,1,1,1,2,2,2,2).$$
\end{lemma}

\begin{proof}
Straightforward from Lemmas \ref{3} and  \ref{lemma4}.
\end{proof}

\begin{theorem}\label{999}
Let $S$ be the set containing 4 copies of the positive integers that are congruent to $\pm 1$ modulo 6, and 8 copies of the positive integers that are congruent to $\pm 2$ modulo 6; let $T$ be the set containing 8 copies of the positive integers that are congruent to $\pm 1$ modulo 6, and 4 copies of the positive integers that are congruent to $\pm 2$ modulo 6. Then, for any $N\geq 2$,
$$D_S(N)=D_T(N-1).$$\end{theorem}

\begin{proof}
Straightforward from Theorem \ref{main} and Lemma \ref{9}.
\end{proof}

\begin{lemma}\label{10}
Condition (i) of Theorem \ref{main} holds for $N_0=4$,  $C=10$, $m=3$, and
$$(A_1,\dots,A_{12})=(2,2,2,2,2,2,4,4,4,4,4,4), (B_1,\dots,B_{12})=(1,1,1,1,1,1,3,3,3,3,3,3).$$
\end{lemma}

\begin{proof}
By Lemmas \ref{4} and \ref{lemma2}, this statement is equivalent to condition (ii) of Lemma \ref{lemma2} holding for $C=30$,
\small
$$(A_1,\dots,A_{12})=(6,6,6,6,6,6,12,12,12,12,12,12), (B_1,\dots,B_{12})=(3,3,3,3,3,3,9,9,9,9,9,9),$$ 
\normalsize
\Small
$$(A'_1,\dots,A'_{12})=(5,5,5,5,5,5,15,15,15,15,15,15), (B'_1,\dots,B'_{12})=(0,0,0,0,0,0,10,10,10,10,10,10).$$
\normalsize
Note that $m=9$, $m'=15$, and $|S'_{k'}|=|S_k|=6$. So, this is equivalent to the statement that $|Q_{N+k'}'|+|R_{N+k}|= |R_{N+k'}'|+|Q_{N+k}|$. Thus, it suffices to show that for each $(d_1,\dots,d_{12})\in Q_{N+k}$ or $(e_1',\dots,e_{12}')\in R_{N+k'}'$ there is a corresponding $(e_1,\dots,e_{12})\in R_{N+k}$ or $(d_1',\dots,d_{12}')\in Q_{N+k'}'$ and vice-versa.

For an arbitrary $d$-tuple, let $w_i=d_i+d_{i+6}$ for each $i$. Clearly, $\sum_{i=1}^{6} w_i$ is odd, so either $1$, $3$, or $5$ of them are odd. If $w_{i_1}$, $w_{i_2}$, and $w_{i_3}$ are odd and $w_{j_1}$, $w_{j_2}$, and $w_{j_3}$ are even, then the map

$$e_{i_x}=d_{i_x}+\frac{1-w_{i_x}-w_{j_x}}{2},\phantom{x}e_{i_x+6}=d_{i_x+6}+\frac{1-w_{i_x}-w_{j_x}}{2},$$$$e_{j_x}=d_{j_x}+\frac{1-w_{i_x}-w_{j_x}}{2},\phantom{x}e_{j_x+6}=d_{j_x+6}+\frac{1-w_{i_x}-w_{j_x}}{2},$$
for $x\in\{1,2,3\}$, yields an $e$-tuple of equal value. This map gives a bijection between the set of $d$-tuples for which three of the $w$'s are odd, and the set of $e$-tuples for which three of the $e_i+e_{i+6}$ are odd. That leaves the cases where one or all but one of them are odd. 

There is an obvious value-preserving bijection from the set of $d$-tuples for which $w_i$ has the opposite parity as the rest, to the set of $d$-tuples for which $w_1$ has the opposite parity as the rest, and there is an obvious bijection from the set of $e$-tuples for which $e_i+e_{i+6}$ has the opposite parity as the rest, to the set of $e$-tuples for which $e_1+e_7$ has the opposite parity as the rest, for each $i$. So, we can focus on the cases where the first one has a different parity than all of the others.

If $w_1$ is odd and the rest are even, then there exist integers $(f_1,\dots,f_{12})$ for which $(d_1,d_7)=(1,0)-f_1(1,1)+f_2(-1,1)$, and $(d_i,d_{i+6})=(0,0)+f_{2i-1}(1,1)+f_{2i}(1,-1)$, for each $2\le i\le 6$. If each $f$-tuple is considered to have a value of
\[60\sum_{i=1}^{12} {f_i \choose 2}+12f_1+6f_2+\sum_{i=2}^6 \Big(18f_{2i-1}+24f_{2i}\Big)+6,\]
then it is easy to check that this map is a value-preserving bijection. If $w_1$ is even and the rest are odd, then there exist integers $(g_1,\dots,g_{12})$ for which $(d_1,d_7)=(0,0)+g_1(1,1)+g_2(1,-1)$, and $(d_i,d_{i+6})=(1,0)-g_{2i-1}(1,1)+g_{2i}(-1,1)$, for each $2\le i\le 6$. If each $g$-tuple is considered to have a value of
\[60\sum_{i=1}^{12} {g_i \choose 2}+18g_1+24g_2+\sum_{i=2}^6 \Big(12g_{2i-1}+6g_{2i}\Big)+30,\]
then this map is a value-preserving bijection.

If $e_1+e_7$ is odd and the rest are even,  there exist integers $(h_1,\dots,h_{12})$ for which $(e_1,e_7)=(1,0)-h_1(1,1)+h_2(-1,1)$, and $(e_i,e_{i+6})=(0,0)+h_{2i-1}(1,1)+h_{2i}(1,-1)$, for each $2\le i\le 6$. If each $h$-tuple is considered to have a value of \[60\sum_{i=1}^{12} {h_i \choose 2}+18h_1+6h_2+\sum_{i=2}^6 \Big(12h_{2i-1}+24h_{2i}\Big)+12,\] then this map is a value-preserving bijection. If $e_1+e_7$ is even and the rest are odd,  there exist integers $(k_1,\dots,k_{12})$ for which $(e_1,e_7)=(0,0)+k_1(1,1)+k_2(1,-1)$, and $(e_i,e_{i+6})=(1,0)-k_{2i-1}(1,1)+k_{2i}(-1,1)$, for each $2\le i\le 6$. 

If each $k$-tuple is considered to have a value of \[60\sum_{i=1}^{12} {k_i \choose 2}+12k_1+24k_2+\sum_{i=2}^6 \Big(18k_{2i-1}+6k_{2i}\Big)+24,\]  this map is a value-preserving bijection.

So, tuples of these types have values of:
\Small
\begin{align*}
&60\sum_{i=1}^{12} {f_i \choose 2}+6f_2+12f_1+18f_3+24f_4+18f_5+18f_7+18f_9+18f_{11} +24f_6+24f_8+24f_{10}+24f_{12}+6,\\
&60\sum_{i=1}^{12} {g_i \choose 2}+6g_4+12g_3+18g_1+24g_2+12g_5+12g_7+12g_9+12g_{11} +6g_6+6g_8+6g_{10}+6g_{12}+30,\\
&60\sum_{i=1}^{12} {h_i \choose 2}+6h_2+12h_3+18h_1+24h_4+12h_5+12h_7+12h_9+12h_{11} +24h_6+24h_8+24h_{10}+24h_{12}+12,\\
&60\sum_{i=1}^{12} {k_i \choose 2}+6k_4+12k_1+18k_3+24k_2+18k_5+18k_7+18k_9+18k_{11} +6k_6+6k_8+6k_{10}+6k_{12}+24.
\end{align*}
\normalsize

Now, let $S$ be the set of all tuples $(s_1,s_2,s_3,s_4)$, and let an arbitrary element of $S$ have a value of
$$60\sum_{i=1}^{4} {s_i \choose 2}+6s_1+12s_2+18s_3+24s_4+6.$$

Also let $Q$ be the set of all tuples $(q_1,q_2,q_3,q_4)$ or $(q_5,q_6,q_7,q_8)$, and let an arbitrary element of $Q$ have a value of $$60\sum_{i=1}^{4} {q_i \choose 2}+ 18q_1+18q_2+18q_3+18q_4, \text{{\ }{\ }or{\ }{\ }} 60\sum_{i=5}^{8} {q_i \choose 2}+ 12q_5+12q_6+12q_7+12q_8+6.$$

Further, regard the second group of tuples as being of the opposite type as the first. Then let $R$ be the set of all tuples $(r_1,r_2,r_3,r_4)$ or $(r_5,r_6,r_7,r_8)$, and let an arbitrary element of $R$ have a value of
$$60\sum_{i=1}^{4} {r_i \choose 2}+ 24r_1+24r_2+24r_3+24r_4,    \text{{\ }{\ }or{\ }{\ }} 60\sum_{i=5}^{8} {r_i \choose 2}+ 6r_5+6r_6+6r_7+6r_8+18.$$

Again, regard the second group of tuples as being of the opposite type as the first. Finally, let $T$ be the set of all ordered triples of an element of $Q$, an element of $R$, and an element of $S$, and let each element of $T$ have a value equal to the sum of its elements' values. An element of $T$ should be considered to be of one type if its elements of $Q$ and $R$ are both of their first types or both of their second types, and of the opposite type if one of them is of its first type and the other is of its second type. 

There is an obvious bijection from the union of the sets of $f$-, $g$-, $h$-, and $k$-tuples to $T$ that preserves both value and type. So, there is a bijection from the set of all $d$- and $e$-tuples that we have not already canceled out to $6$ copies of $T$ that preserves value and type.

Similarly, for an arbitrary $d'$-tuple, let $w'_i=d'_i+d'_{i+6}$, for each $i$. The same maps we used before form a value-preserving bijection between the set of all $d'$-tuples such that exactly $3$ of the $w_i'$ are odd and the set of all $e'$-tuples such that exactly $3$ of the  $e'_i+e'_{i+6}$ are odd. With those cases eliminated, we can focus on the case where $w_1'$ or $e_1'+e_7'$ is the one having the opposite parity as the others, for the same reasons as before. 

For each $i$, define $f_i',g_i',h_i', k_i'$ analogously to the way we defined $f_i,g_i,h_i, k_i$. The map $(d_1',d_7')=(1,0)-f_1'(1,1)+f_2'(-1,1)$, $(d_i',d_{i+6}')=(0,0)+f_{2i-1}'(1,1)+f_{2i}'(1,-1)$ is a value-preserving bijection if $(f_1',\dots,f_{12}')$ is considered to have a value of
\[60\sum_{i=1}^{12} {f_i' \choose 2}+10f_1'+10f_2'+\sum_{i=2}^6 \Big(20f_{2i-1}'+20f_{2i}'\Big)+6.\]

The map $(d_1',d_7')=(0,0)+g_1'(1,1)+g_2'(1,-1)$, $(d_i',d_{i+6}')=(1,0)-g_{2i-1}'(1,1)+g_{2i}'(-1,1)$ is a value-preserving bijection if $(g_1',\dots,g_{12}')$ is considered to have a value of \[60\sum_{i=1}^{12} {g_i' \choose 2}+20g_1'+20g_2'+\sum_{i=2}^6 \Big(10g_{2i-1}'+10g_{2i}'\Big)+26.\] 

The map $(e_1',e_7')=(1,0)-h_1'(1,1)+h_2'(-1,1)$, $(e_i',e_{i+6}')=(0,0)+h_{2i-1}'(1,1)+h_{2i}'(1,-1)$ is a value-preserving bijection if $(h_1',\dots,h_{12}')$ is considered to have a value of \[60\sum_{i=1}^{12} {h_i' \choose 2}+20h_1'+10h_2'+\sum_{i=2}^6 \Big(10h_{2i-1}'+20h_{2i}'\Big)+16.\]

Finally, the map $(e_1',e_7')=(0,0)+k_1'(1,1)+k_2'(1,-1)$, $(e_i',e_{i+6}')=(1,0)-k_{2i-1}'(1,1)+k_{2i}'(-1,1)$ is a value-preserving bijection if $(k_1',\dots,k_{12}')$ is considered to have a value of \[60\sum_{i=1}^{12} {k_i' \choose 2}+10k_1'+20k_2'+\sum_{i=2}^6 \Big(20k_{2i-1}'+10k_{2i}'\Big)+16.\]  

Now, let $S'$ be the set of all tuples $(s_1',s_2',s_3',s_4')$, and let an arbitrary element of $S'$ have a value of $$60\sum_{i=1}^{4} {s_i' \choose 2}+10s_1'+10s_2'+20s_3'+20s_4'+6.$$ 

Also let $Q'$ be the set of all tuples $(q_1',q_2',q_3',q_4')$ or $(q_5',q_6',q_7',q_8')$, and let an arbitrary element of $Q'$ have a value of $$60\sum_{i=1}^{4} {q_i' \choose 2}+ 20q_1'+20q_2'+20q_3'+20q_4', \text{{\ }{\ }or{\ }{\ }}60\sum_{i=5}^{8} {q_i' \choose 2}+ 10q_5'+10q_6'+10q_7'+10q_8'+10.$$

Regard the second group of tuples as being of the opposite type as the first. Then let $R'$ be the set of all tuples $(r_1',r_2',r_3',r_4')$ or $(r_5',r_6',r_7',r_8')$, and let an arbitrary element of $R'$ have a value of $$60\sum_{i=1}^{4} {r_i' \choose 2}+ 20r_1'+20r_2'+20r_3'+20r_4', \text{{\ }{\ }or{\ }{\ }}60\sum_{i=5}^{8} {r_i' \choose 2}+ 10r_5'+10r_6'+10r_7'+10r_8'+10.$$

Again, regard the second group of tuples as being of the opposite type as the first. Finally, let $T'$ be the set of all ordered triples of an element of $Q'$, an element of $R'$ and an element of $S'$, and let each element of $T'$ have a value equal to the sum of its elements' values. An element of $T'$ should be considered to be of one type if its elements of $Q'$ and $R'$ are both of their first types or both of their second types, and of the opposite type if one of them is of its first type and the other is of its second type. 

There is an obvious bijection from the union of the sets of $f'$-, $g'$-, $h'$- and $k'$-tuples to $T'$ that preserves both value and type. So, there exists a bijection from the set of all $d'$- and $e'$-tuples that we have not already canceled out to $6$ copies of $T'$, which preserves value and type. Therefore, in order to show the lemma it now suffices to prove that there is a value-preserving bijection from the set of elements of $T\cup T'$ of one type, to the set of elements of $T\cup T'$ of the opposite type.

Now, consider the following map in $Q$: $q_i=q_{i+4}-\frac{q_5+q_6+q_7+q_8-1}{2}.$

This map is a value-preserving bijection from the set of tuples $(q_5,q_6,q_7,q_8)$ with odd sums to the set of tuples $(q_1,q_2,q_3,q_4)$ with odd sums. Furthermore, it maps all tuples $(q_5,q_6,q_7,q_8)$ with even sums to tuples of half-integers $(q_1,q_2,q_3,q_4)$ with even sums and equal values. At this point, the map
$$a_1=\frac{q_1+q_2+q_3+q_4}{2},a_2=\frac{q_1+q_2-q_3-q_4}{2},a_3=\frac{q_1-q_2+q_3-q_4}{2},a_4=\frac{q_1-q_2-q_3+q_4}{2}$$
is a value-preserving bijection from the set of tuples $(q_1,q_2,q_3,q_4)$ that results from the last map to the set of tuples $(a_1,a_2,a_3,a_4)$, if $(a_1,a_2,a_3,a_4)$ is considered to have a value of
$$60\sum_{i=1}^{4} {a_i \choose 2}+6a_1+30a_2+30a_3+30a_4.$$

It also preserves the type if these tuples are considered to have a type based on the parity of $\sum_{i=1}^{4}a_i$.

The same pair of maps cancels all tuples in $R$ with odd sums and maps the rest to tuples $(b_1,b_2,b_3,b_4)$ with the same value and type, if $(b_1,b_2,b_3,b_4)$ is considered to have a value of
$$60\sum_{i=1}^{4} {b_i \choose 2}+18b_1+30b_2+30b_3+30b_4,$$
and a type dependent on the parity of $\sum_{i=1}^{4}b_i$.

They also cancel all tuples in $Q'$ with odd sums and map the rest to tuples $(a_1',a_2',a_3',a_4')$ with the same value and type, if $(a_1',a_2',a_3',a_4')$ is considered to have a value of
$$60\sum_{i=1}^{4} {a_i' \choose 2}+10a_1'+30a_2'+30a_3'+30a_4',$$
and a type dependent on the parity of $\sum_{i=1}^{4}a_i'$.

Finally, they cancel all tuples in $R'$ with odd sums and map the rest to tuples $(b_1',b_2',b_3',b_4')$ with the same value and type, if $(b_1',b_2',b_3',b_4')$ is considered to have a value of
$$60\sum_{i=1}^{4} {b_i' \choose 2}+10b_1'+30b_2'+30b_3'+30b_4',$$
and a type dependent on the parity of $\sum_{i=1}^{4}b_i'$.

Therefore, to prove the result it suffices to find a bijection from the union of the set of all tuples $(a_1,a_2,a_3,a_4,b_1,b_2,b_3,b_4,s_1,s_2,s_3,s_4)$ for which $\sum_{i=1}^{4}(a_i+b_i)\equiv 0\pmod{2}$ and the set of all tuples $(a_1',a_2',a_3',a_4',b_1',b_2',b_3',b_4',s_1',s_2',s_3',s_4')$ for which $\sum_{i=1}^{4}(a_i'+b_i')\equiv 1\pmod{2}$,  to the union of the set of all tuples $(a_1,a_2,a_3,a_4,b_1,b_2,b_3,b_4,s_1,s_2,s_3,s_4)$ for which $\sum_{i=1}^{4}(a_i+b_i)\equiv 1\pmod{2}$ and the set of all tuples $(a_1',a_2',a_3',a_4',b_1',b_2',b_3',b_4',s_1',s_2',s_3',s_4')$ for which $\sum_{i=1}^{4}(a_i'+b_i')\equiv 0\pmod{2}$, which preserves the value of 

\scriptsize
\begin{align*}
&60\sum_{i=1}^4 \left({a_i \choose 2}+{b_i \choose 2}+{s_i \choose 2}\right)+6a_1+30a_2+30a_3+30a_4+18b_1+30b_2+30b_3+30b_4+6s_1+12s_2+18s_3+24s_4
\end{align*}
\normalsize
or
\scriptsize
\begin{align*}
&60\sum_{i=1}^4 \left({a_i' \choose 2}+{b_i' \choose 2}+{s_i' \choose 2}\right)+10a_1'+30a_2'+30a_3'+30a_4'+10b_1'+30b_2'+30b_3'+30b_4'+10s_1'+10s_2'+20s_3'+20s_4'.
\end{align*}
\normalsize

Furthermore, notice that $(a_2',a_3',a_4',b_2',b_3',b_4')$ has exactly the same effect on the value and type of its tuple as $(a_2,a_3,a_4,b_2,b_3,b_4)$ does. Hence, if we can find a bijection from the union of the set of all tuples $(a_1,b_1,s_1,s_2,s_3,s_4)$ for which $a_1+b_1\equiv 0\pmod{2}$ and the set of all tuples $(a_1',b_1',s_1',s_2',s_3',s_4')$ for which $a_1'+b_1'\equiv 1\pmod{2}$,  to the union of the set of all tuples $(a_1,b_1,s_1,s_2,s_3,s_4)$ for which $a_1+b_1\equiv 1\pmod{2}$ and the set of all tuples $(a_1',b_1',s_1',s_2',s_3',s_4')$ for which $a_1'+b_1'\equiv 0\pmod{2}$, which preserves the value of 
\footnotesize 
\begin{align*}
&60\sum_{i=1}^4 {s_i \choose 2}+60{a_1 \choose 2}+60{b_1 \choose 2}+6a_1+18b_1+6s_1+12s_2+18s_3+24s_4\text{ or}\\
&60\sum_{i=1}^4 {s_i' \choose 2}+60{a_1' \choose 2}+60{b_1' \choose 2}+10a_1'+10b_1'+10s_1'+10s_2'+20s_3'+20s_4',
\end{align*}
\normalsize
then we can extend that bijection to the desired bijection by having it act as the identity on $(a_2,a_3,a_4,b_2,b_3,b_4)$ or $(a_2',a_3',a_4',b_2',b_3',b_4')$.

Now, observe that if $a_1$ and $s_1$, $b_1$ and $s_3$, $a_1'$ and $s_1'$, or $b_1'$ and $s_2'$ have different parities, then we can switch them to get a tuple with the same value but the opposite type. This allows us to cancel all such tuples. Given $(a_1,b_1,s_1,s_2,s_3,s_4)$ such that $a_1$ and $s_1$ have the same parity and $b_1$ and $s_3$ have the same parity, there must exist $z_1, z_2, z_3, z_4\in \mathbb{Z}$ such that $(a_1,s_1)=z_1(1,1)+z_2(1,-1)$ and $(b_1,s_3)=z_3(1,1)+z_4(1,-1)$. 

Also, the value of this tuple is equal to \[120\sum_{i=1}^4 {z_i \choose 2}+60{s_2 \choose 2}+60{s_4 \choose 2}+12z_1+60z_2+36z_3+60z_4+12s_2+24s_4,\] and its type depends on whether $\sum_{i=1}^{4} z_i$ is even or odd.
	
Similarly, given $(a_1',b_1',s_1',s_2',s_3',s_4')$ such that $a_1'$ and $s_1'$ have the same parity and $b_1'$ and $s_2'$ have the same parity, there must exist $z_1', z_2', z_3', z_4'\in \mathbb{Z}$ such that $(a_1',s_1')=z_1'(1,1)+z_2'(1,-1)$ and $(b_1',s_2')=z_3'(1,1)+z_4'(1,-1)$. The value of this tuple is \[120\sum_{i=1}^4 {z_i' \choose 2}+60{s_3' \choose 2}+60{s_4' \choose 2}+20z_1'+60z_2'+20z_3'+60z_4'+20s_3'+20s_4',\] and its type depends on whether $\sum_{i=1}^{4} z_i'$ is even or odd.

It is easy to see that, for arbitrary $(s_2,z_1)$, there exists exactly one of the following: $x_1,x_2\in \mathbb{Z}$ such that $(s_2,z_1)=x_1(-1,1)+x_2(2,1)$; $x_3,x_4\in \mathbb{Z}$ such that $(s_2,z_1)=(-1,0)-x_3(-1,1)+(1-x_4)(2,1)$; or $x_5,x_6\in \mathbb{Z}$ such that $(s_2,z_1)=(1,0)+x_5(-1,1)+x_6(2,1)$. 

Also,
$$120{z_1 \choose 2}+60{s_2 \choose 2}+12z_1+12s_2= 180{x_1 \choose 2}+360{x_2 \choose 2}+60x_1+96x_2,$$
$$\text{{\ }{\ }or{\ }{\ }}180{x_3 \choose 2}+360{x_4 \choose 2}+60x_3+24x_4+24,\text{{\ }{\ }or{\ }{\ }}180{x_5 \choose 2}+360{x_6 \choose 2}+0x_5+216x_6+12.$$

We have $z_1\equiv x_1+x_2\text{, }x_3+x_4+1\text{, or }x_5+x_6\pmod{2}$. Note that no matter what the other variables are, replacing $x_5$ with $1-x_5$ will always invert the tuple's type without affecting its value, so that map cancels out all tuples resulting from that case.

Similarly, for arbitrary $(s_4,z_3)$, there exists exactly one of the following: $y_1,y_2\in \mathbb{Z}$ such that $(-s_4,z_3)=y_1(-1,1)+y_2(2,1)$;  $y_3,y_4\in \mathbb{Z}$ such that $(-s_4,z_3)=(-1,0)-y_3(-1,1)+y_4(2,1)$; or $y_5,y_6\in \mathbb{Z}$ such that $(-s_4,z_3)=(1,0)+y_5(-1,1)+y_6(2,1)$.

Also,
$$120{z_3 \choose 2}+60{s_4 \choose 2}+36z_3+24s_4=180{y_1 \choose 2}+360{y_2 \choose 2}+60y_1+168y_2,$$
$$\text{{\ }{\ }or{\ }{\ }}180{y_3 \choose 2}+360{y_4 \choose 2}+60y_3+48y_4+24, \text{{\ }{\ }or{\ }{\ }}180{y_5 \choose 2}+360{y_6 \choose 2}+0y_5+288y_6+36.$$

We have $z_3\equiv y_1+y_2\text{, }y_3+y_4\text{, or } y_5+y_6\pmod{2}$. In the third case, replacing $y_5$ with $1-y_5$ will always invert the tuple's type without affecting its value, so this map cancels out all tuples resulting from that case.

Likewise, for arbitrary $(s_3',z_1')$ there exists exactly one of the following: $x_1',x_2'\in \mathbb{Z}$ such that $(s_3',z_1')=x_1'(-1,1)+x_2'(2,1)$;  $x_3',x_4'\in \mathbb{Z}$ such that $(s_3',z_1')=(-1,0)+x_3'(-1,1)+x_4'(2,1)$; or $x_5',x_6'\in \mathbb{Z}$ such that $(s_3',z_1')=(1,0)+x_5'(-1,1)+x_6'(2,1)$. 

Also,
$$120{z_1' \choose 2}+60{s_3' \choose 2}+20z_1'+20s_3'=180{x_1' \choose 2}+360{x_2' \choose 2}+60x_1'+120x_2',$$
$$\text{{\ }{\ }or{\ }{\ }}180{x_3' \choose 2}+360{x_4' \choose 2}+120x_3'+0x_4'+40, \text{{\ }{\ }or{\ }{\ }}180{x_5' \choose 2}+360{x_6' \choose 2}+0x_5'+240x_6'+20.$$

We have $z_1'\equiv x_1'+x_2'\text{, }x_3'+x_4'\text{, or }x_5'+x_6'\pmod{2}$. In the third case, replacing $x_5'$ with $1-x_5'$ will always invert the tuple's type without affecting its value, while in the second case, replacing $x_4'$ by $1-x_4'$ will always invert the tuple's type without affecting its value. So, the only case that does not cancel itself out is the first.

Notice that $(s_4',z_3')$ has exactly the same effect on the value and type of the tuple as $(s_3',z_1')$, so $(s_4',z_3')$ can also be expressed in exactly one of the forms: $y_1'(-1,1)+y_2'(2,1)$, $(-1,0)+y_3'(-1,1)+y_4'(2,1)$, or $(1,0)+y_5'(-1,1)+y_6'(2,1)$. The maps $y_4''= 1-y_4'$ and $y_5''= 1-y_5'$ still cancel out all tuples covered by the second and third cases, and in the first case, we have
\[120{z_3' \choose 2}+60{s_4' \choose 2}+20z_3'+20s_4'=180{y_1' \choose 2}+360{y_2' \choose 2}+60y_1'+120y_2'\] and $z_3'\equiv y_1'+y_2' \pmod{2}$.

So, any tuple that has not been canceled out by now has a value of whichever of the following is defined:
\scriptsize
\begin{align*}
&120{z_2 \choose 2}+120{z_4 \choose 2}+180{x_1 \choose 2}+360{x_2 \choose 2}+180{y_1 \choose 2}+360{y_2 \choose 2}+60z_2+60z_4+60x_1+96x_2+60y_1+168y_2,\\
&120{z_2 \choose 2}+120{z_4 \choose 2}+180{x_3 \choose 2}+360{x_4 \choose 2}+180{y_1 \choose 2}+360{y_2 \choose 2}+60z_2+60z_4+60x_3+24x_4+60y_1+168y_2+24,\\
&120{z_2 \choose 2}+120{z_4 \choose 2}+180{x_1 \choose 2}+360{x_2 \choose 2}+180{y_3 \choose 2}+360{y_4 \choose 2}+60z_2+60z_4+60x_1+96x_2+60y_3+48y_4+24,\\
&120{z_2 \choose 2}+120{z_4 \choose 2}+180{x_3 \choose 2}+360{x_4 \choose 2}+180{y_3 \choose 2}+360{y_4 \choose 2}+60z_2+60z_4+60x_3+24x_4+60y_3+48y_4+48,\\
&120{z_2' \choose 2}+120{z_4' \choose 2}+180{x_1' \choose 2}+360{x_2' \choose 2}+180{y_1' \choose 2}+360{y_2' \choose 2}+60z_2'+60z_4'+60x_1'+120x_2'+60y_1'+120y_2'.
\end{align*}
\normalsize

The tuple's type depends on whether $z_2+z_4+x_1+x_2+y_1+y_2$, $z_2+z_4+x_3+x_4+y_1+y_2$+1, $z_2+z_4+x_1+x_2+y_3+y_4$, $z_2+z_4+x_3+x_4+y_3+y_4+1$, or $z_2'+z_4'+x_1'+x_2'+y_1'+y_2'+1$ is even.

Notice that whichever of $(z_2,z_4,x_1,y_1)$, $(z_2,z_4,x_3,y_1)$, $(z_2,z_4,x_1,y_3)$, $(z_2,z_4,x_3,y_3)$, or $(z_2',z_4',x_1',y_1')$ is defined has the same effect on the value and type of the tuple in every case. So, if we can find a bijection between the set of all tuples $(x_2,y_2)$, $(x_4,y_2)$, $(x_2,y_4)$, $(x_4,y_4)$, or $(x_2',y_2')$ for which $x_2+y_2$, $x_4+y_2+1$, $x_2+y_4$, $x_4+y_4+1$, or $x_2'+y_2'+1$ is even to the set of tuples of any of these types for which it is odd, which preserves 
$$360{x_2 \choose 2}+360{y_2 \choose 2}+96x_2+168y_2,\phantom{x}360{x_4 \choose 2}+360{y_2 \choose 2}+24x_4+168y_2+24,$$
$$360{x_2 \choose 2}+360{y_4 \choose 2}+96x_2+48y_4+24,\phantom{x}360{x_4 \choose 2}+360{y_4 \choose 2}+24x_4+48y_4+48,$$
$$\text{ or }\phantom{x}360{x_2' \choose 2}+360{y_2' \choose 2}+120x_2'+120y_2',$$
then we can extend it to a value-preserving bijection from the set of all remaining tuples of one type to the set of all remaining tuples of the other type, by having it leave $(z_2,z_4,x_1,y_1)$ or its equivalent unchanged.

It is easy to see that any pair of integers can be expressed in exactly one of the forms $u(1,2)+v(2,-1)$, $(-1,0)+u(1,2)+v(2,-1)$, $(1,0)+u(1,2)+v(2,-1)$, $(0,-1)+u(1,2)+v(2,-1)$, or $(0,1)+u(1,2)+v(2,-1)$, with $u$ and $v$ integers. 

For arbitrary integers $u$ and $v$, each of the following pairs of tuples has the same value and opposite types:
\begin{align*}
&(x_2,y_2)=u(2,1)+v(-1,2),(x_2',y_2')=u(1,2)+v(-2,1)\\
&(x_2,y_2)=(1,0)+u(2,1)+v(-1,2),(x_4,y_4)=(0,1)+u(-2,1)+v(-1,-2)\\
&(x_2,y_2)=(0,-1)+u(2,1)+v(-1,2),(x_4,y_2)=(0,1)+u(-1,-2)+v(2,-1)\\
&(x_2,y_2)=(-1,0)+u(2,1)+v(-1,2),(x_2,y_2)=(1,1)+u(-2,-1)+v(-1,2)\\
&(x_2,y_2)=(0,1)+u(2,1)+v(-1,2),(x_2,y_4)=(1,1)+u(-1,2)+v(2,1)\\
&(x_2,y_4)=u(2,1)+v(-1,2),(x_4,y_2)=u(2,-1)+v(1,2)\\
&(x_2,y_4)=(1,0)+u(2,1)+v(-1,2),(x_2',y_2')=(0,1)+u(-2,1)+v(-1,-2)\\
&(x_2,y_4)=(0,-1)+u(2,1)+v(-1,2),(x_2,y_4)=(-1,1)+u(2,1)+v(1,-2)\\
&(x_2,y_4)=(0,1)+u(2,1)+v(-1,2),(x_4,y_4)=(1,0)+u(-1,-2)+v(2,-1)\\
&(x_4,y_2)=(1,0)+u(1,2)+v(2,-1),(x_4,y_4)=u(-2,1)+v(-1,-2)\\
&(x_4,y_2)=(1,-1)+u(1,2)+v(2,-1),(x_2',y_2')=(-1,0)+u(1,2)+v(-2,1)\\
&(x_4,y_2)=(0,-1)+u(1,2)+v(2,-1),(x_4,y_2)=(1,1)+u(-1,-2)+v(2,-1)\\
&(x_4,y_4)=(-1,0)+u(1,2)+v(2,-1),(x_4,y_4)=(1,-1)+u(1,2)+v(-2,1)\\
&(x_4,y_4)=(1,1)+u(1,2)+v(2,-1),(x_2',y_2')=(1,0)+u(2,-1)+v(1,2)\\
&(x_2',y_2')=(0,-1)+u(1,2)+v(2,-1),(x_2',y_2')=(1,1)+u(-1,-2)+v(2,-1).
\end{align*}

Each tuple of the form $(x_2,y_2)$, $(x_2,y_4)$, $(x_4,y_2)$, $(x_4,y_4)$, or $(x_2',y_2')$ is stated to have the same value as another tuple by exactly one of these. The only time a tuple shows up more than once on the same line is if it can be expressed in the forms on each side, and all such lines are  involutions. Therefore, these equalities combine to yield a value-preserving bijection from the tuples of one type to the tuples of the other type. We have already shown that this is sufficient to prove the lemma.
\end{proof}

\begin{theorem}\label{1000}
Let $S$ be the set containing 6 copies of the even positive integers that are not multiples of 5, and $T$ the set containing 6 copies of the odd  positive integers that are not multiples of 5. Then, for any $N\geq 4$,
$$D_S(N)=D_T(N-3).$$
\end{theorem}

\begin{proof}
Straightforward from Theorem \ref{main} and Lemma \ref{10}.
\end{proof}

\begin{lemma}\label{11}
Condition (i) of Theorem \ref{main} holds for $N_0=2$,  $C=10$, $m=1$, and
$$(A_1,\dots,A_{12})=(1,1,2,2,2,2,3,3,4,4,4,4), (B_1,\dots,B_{12})=(1,1,1,1,2,2,3,3,3,3,4,4).$$
\end{lemma}

\begin{proof}
Straightforward from Lemmas \ref{10} and \ref{lemma4}.
\end{proof}

\begin{theorem}\label{1001}
Let $S$ be the set containing 2 copies of the odd positive integers that are not multiples of 5, and 4 copies of the even positive integers that are not multiples of 5; let $T$ be the set containing 2 copies of the even positive integers that are not multiples of 5, and 4 copies of the odd positive integers that are not multiples of 5.  Then, for any $N\geq 2$,
$$D_S(N)=D_T(N-1).$$
\end{theorem}

\begin{proof}
Straightforward from Theorem \ref{main} and Lemma \ref{11}.
\end{proof}

Finally, we present  a large sample of further interesting colored partition identities that we conjecture to be true. We list  as conjectures the equations  corresponding bijectively to these partition identities via Theorem \ref{main}. We have verified them for  $N$ up to $2000$, by means of a computer program.

\begin{conjecture}\label{c27}
Condition (i) of Theorem \ref{main} holds for $N_0= 4$,  $C=50$, $m=3$, and
\Small
$$(A_1,\dots,A_{12})=(2,4,6,8,10,12,14,16,18,20,22,24), (B_1,\dots,B_{12})=(1,3,5,7,9,11,13,15,17,19,21,23).$$
\normalsize
\end{conjecture}

\noindent
\textbf{Corollary  to Conjecture \ref{c27}.}
\emph{Let $S$ be the set containing one copy of the even positive integers that are not multiples of 25, and $T$ the set containing one copy of the odd positive integers that are not multiples of 25. Then, for any $N\geq 4$,
$$D_S(N)=D_T(N-3).$$}

\begin{conjecture}\label{c23}
Condition (i) of Theorem \ref{main} holds for $N_0= 4$,  $C=26$, $m=3$, and
$$(A_1,\dots,A_{12})=(2,2,4,4,6,6,8,8,10,10,12,12), (B_1,\dots,B_{12})=(1,1,3,3,5,5,7,7,9,9,11,11).$$
\end{conjecture}

\noindent
\textbf{Corollary  to Conjecture \ref{c23}.}
\emph{Let $S$ be the set containing 2 copies of the even positive integers that are not multiples of 13, and $T$  the set containing  2 copies of the odd positive integers that are not multiples of 13. Then, for any $N\geq 4$,
$$D_S(N)=D_T(N-3).$$}

\begin{conjecture}\label{c3}
Condition (i) of Theorem \ref{main} holds for $N_0= 3$,  $C=14$, $m=3$, and
$$(A_1,\dots,A_{12})=(1,1,1,3,3,3,5,5,5,7,7,7), (B_1,\dots,B_{12})=(0,0,0,2,2,2,4,4,4,6,6,6).$$
\end{conjecture}

\noindent
\textbf{Corollary  to Conjecture \ref{c3}.}
\emph{Let $S$ be the set containing  3 copies of the odd positive integers and 3 more copies of the odd positive multiples of 7, and  $T$  the set containing 3 copies of the even positive integers and 3 more copies of the positive multiples of 14. Then, for any $N\geq 3$,
$$D_S(N)=4D_T(N-3).$$}

\begin{conjecture}\label{c4}
Condition (i) of Theorem \ref{main} holds for $N_0= 4$,  $C=14$, $m=3$, and
$$(A_1,\dots,A_{12})=(2,2,2,2,4,4,4,4,6,6,6,6), (B_1,\dots,B_{12})=(1,1,1,1,3,3,3,3,5,5,5,5).$$
\end{conjecture}

\noindent
\textbf{Corollary  to Conjecture \ref{c4}.}
\emph{Let $S$ be the set containing  4 copies of the even positive integers that are not multiples of 7, and $T$ the set containing 4 copies of the odd positive integers that are not multiples of 7. Then, for any $N\geq 4$,
$$D_S(N)=D_T(N-3).$$}

\begin{conjecture}\label{c10}
Condition (i) of Theorem \ref{main} holds for $N_0= 4$,  $C=18$, $m=3$, and
$$(A_1,\dots,A_{12})=(2,2,2,4,4,4,6,6,6,8,8,8), (B_1,\dots,B_{12})=(1,1,1,3,3,3,5,5,5,7,7,7).$$
\end{conjecture}

\noindent
\textbf{Corollary  to Conjecture \ref{c10}.}
\emph{Let $S$ be the set containing  3 copies of the even positive integers that are not multiples of 9, and $T$  the set containing  3 copies of the odd positive integers that are not multiples of 9. Then, for any $N\geq 4$,
$$D_S(N)=D_T(N-3).$$}

\begin{conjecture}\label{c20}
Condition (i) of Theorem \ref{main} holds for $N_0= 3$,  $C=12$, $m=2$, and
$$(A_1,\dots,A_{12})=(1,1,1,1,4,4,4,4,5,5,5,5), (B_1,\dots,B_{12})=(1,1,1,1,2,2,2,2,5,5,5,5).$$
\end{conjecture}

\noindent
\textbf{Corollary  to Conjecture \ref{c20}.}
\emph{Let $S$ be the set containing 4 copies of  the positive integers that are either congruent  to $\pm 1$ modulo 6 or to $\pm 4$ modulo 12, and $T$ the set containing 4 copies of  the positive integers that are either congruent  to $\pm 1$ modulo 6 or to $\pm 2$ modulo 12. Then, for any $N\geq 3$,
$$D_S(N)=D_T(N-2).$$}

\begin{conjecture}\label{c26}
Condition (i) of Theorem \ref{main} holds for $N_0=2$,  $C=18$, $m=1$, and
$$(A_1,\dots,A_{12})=(1,2,2,3,4,4,5,6,6,7,8,8), (B_1,\dots,B_{12})=(1,1,2,3,3,4,5,5,6,7,7,8).$$
\end{conjecture}

\noindent
\textbf{Corollary  to Conjecture \ref{c26}.}
\emph{Let $S$ be the set containing one copy of the odd positive integers that are not multiples of 9 and 2 copies  of the even positive integers that are not multiples of 9, and $T$  the set containing  2 copies of the odd positive integers that are not multiples of 9 and one copy  of the even positive integers that are not multiples of 9. Then, for any $N\geq 2$,
$$D_S(N)=D_T(N-1).$$}

\begin{conjecture}\label{c5}
Condition (i) of Theorem \ref{main} holds for $N_0= 2$,  $C=8$, $m=2$, and
$$(A_1,\dots,A_{12})=(1,1,1,1,2,2,3,3,3,3,4,4), (B_1,\dots,B_{12})=(0,0,1,1,1,1,2,2,3,3,3,3).$$
\end{conjecture}

\noindent
\textbf{Corollary  to Conjecture \ref{c5}.}
\emph{Let $S$ be the set containing  4 copies of the positive integers that are either odd or congruent to 4 modulo 8, and 2 copies of the positive integers that are  congruent to 2 modulo 4; let $T$ be the set containing 4 copies of the positive integers that are either odd or multiples of 8, and 2 copies of the positive integers that are  congruent to 2 modulo 4. Then, for any $N\geq 2$,
$$D_S(N)=2D_T(N-2).$$}

\begin{conjecture}\label{c6}
Condition (i) of Theorem \ref{main} holds for $N_0= 1$,  $C=8$, $m=1$, and
$$(A_1,\dots,A_{12})=(0,1,1,2,2,2,2,2,2,2,3,3), (B_1,\dots,B_{12})=(0,0,1,1,1,1,2,2,3,3,3,3).$$
\end{conjecture}

\noindent
\textbf{Corollary  to Conjecture \ref{c6}.}
\emph{Let $S$ be the set containing  2 copies of the positive integers that are either odd or multiples of 8, and 7 copies of the positive integers that are  congruent to 2 modulo 4; let $T$ be the set containing 4 copies of the positive integers that are either odd or multiples of 8, and 2 copies of the positive integers that are  congruent to 2 modulo 4. Then, for any $N\geq 1$,
$$D_S(N)=2D_T(N-1).$$}

\begin{conjecture}\label{c6'}
Condition (i) of Theorem \ref{main} holds for $N_0= 2$,  $C=8$, $m=1$, and
$$(A_1,\dots,A_{12})=(1,1,1,1,2,2,3,3,3,3,4,4), (B_1,\dots,B_{12})=(0,1,1,2,2,2,2,2,2,2,3,3).$$
\end{conjecture}

\noindent
\textbf{Corollary  to Conjecture \ref{c6'}.}
\emph{Let $S$ be the set containing  4 copies of the positive integers that are either odd or congruent to 4 modulo 8, and 2 copies of the positive integers that are  congruent to 2 modulo 4; let $T$ be the set containing 2 copies of the positive integers that are either odd or multiples of 8, and 7 copies of the positive integers that are  congruent to 2 modulo 4. Then, for any $N\geq 2$,
$$D_S(N)=D_T(N-1).$$}

\begin{conjecture}\label{c2}
Condition (i) of Theorem \ref{main} holds for $N_0= 1$,  $C=6$, $m=1$, and
$$(A_1,\dots,A_{12})=(0,0,0,1,2,2,2,2,2,3,3,3), (B_1,\dots,B_{12})=(0,0,0,1,1,1,1,1,2,3,3,3).$$
\end{conjecture}

\noindent
\textbf{Corollary  to Conjecture \ref{c2}.}
\emph{Let $S$ be the set containing  one copy of the positive integers congruent to $\pm 1$ modulo 6, 5 copies of the positive integers congruent to $\pm 2$ modulo 6, and 6 copies of the positive multiples of 3; let $T$ be the set containing  5 copies of the positive integers congruent to $\pm 1$ modulo 6, one copy of the positive integers congruent to $\pm 2$ modulo 6, and 6 copies of the positive multiples of 3. Then, for any $N\geq 1$,
$$D_S(N)=D_T(N-1).$$}

\begin{conjecture}\label{c7}
Condition (i) of Theorem \ref{main} holds for $N_0= 1$,  $C=8$, $m=1$, and
$$(A_1,\dots,A_{12})=(0,0,1,1,2,2,2,3,3,4,4,4), (B_1,\dots,B_{12})=(0,0,0,1,1,2,2,2,3,3,4,4).$$
\end{conjecture}

\noindent
\textbf{Corollary  to Conjecture \ref{c7}.}
\emph{Let $S$ be the set containing 2 copies of the odd positive integers,  3 copies of the positive integers that are  congruent to 2 modulo 4, 6 copies of the positive integers that are  congruent to 4 modulo 8, and 4 copies of the positive multiples of 8; let $T$ be the set containing 2 copies of the odd positive integers,  3 copies of the positive integers that are  congruent to 2 modulo 4, 4 copies of the positive integers that are  congruent to 4 modulo 8, and 6 copies of the positive multiples of 8. Then, for any $N\geq 1$,
$$D_S(N)=2D_T(N-1).$$}

\begin{conjecture}\label{c8}
Condition (i) of Theorem \ref{main} holds for $N_0= 1$,  $C=8$, $m=1$, and
$$(A_1,\dots,A_{12})=(0,1,1,1,1,2,2,3,3,3,3,4), (B_1,\dots,B_{12})=(0,0,0,1,1,2,2,2,3,3,4,4).$$
\end{conjecture}

\noindent
\textbf{Corollary  to Conjecture \ref{c8}.}
\emph{Let $S$ be the set containing 2 copies of the  positive integers and 2 more copies of the odd positive integers; let $T$ be the set containing 2 copies of the odd positive integers,  3 copies of the positive integers that are  congruent to 2 modulo 4, 4 copies of the positive integers that are  congruent to 4 modulo 8, and 6 copies of the positive multiples of 8. Then, for any $N\geq 1$,
$$D_S(N)=4D_T(N-1).$$}

\begin{conjecture}\label{c8'}
Condition (i) of Theorem \ref{main} holds for $N_0= 1$,  $C=8$, $m=0$, and
$$(A_1,\dots,A_{12})=(0,0,1,1,2,2,2,3,3,4,4,4), (B_1,\dots,B_{12})=(0,1,1,1,1,2,2,3,3,3,3,4).$$
\end{conjecture}

\noindent
\textbf{Corollary  to Conjecture \ref{c8'}.}
\emph{Let $S$ be the set containing 2 copies of the odd positive integers,  3 copies of the positive integers that are  congruent to 2 modulo 4, 6 copies of the positive integers that are  congruent to 4 modulo 8, and 4 copies of the positive multiples of 8; let $T$ be the set containing 2 copies of the  positive integers and 2 more copies of the odd positive integers. Then, for any $N\geq 1$,
$$D_S(N)=\frac{1}{2}D_T(N).$$}

\begin{conjecture}\label{c9}
Condition (i) of Theorem \ref{main} holds for $N_0= 3$,  $C=18$, $m=3$, and
$$(A_1,\dots,A_{12})=(1,1,1,3,5,5,5,7,7,7,9,9), (B_1,\dots,B_{12})=(0,0,2,2,2,4,4,4,6,8,8,8).$$
\end{conjecture}

\noindent
\textbf{Corollary  to Conjecture \ref{c9}.}
\emph{Let $S$ be the set containing  3 copies of the odd positive integers that are not multiples of 3, one copy of the odd positive multiples of 3 that are not multiples of 9,  and 4 copies of the odd positive multiples of 9; let $T$ be the set containing  3 copies of the even positive integers that are not multiples of 3, one copy of the  positive multiples of 6 that are not multiples of 18,  and 4 copies of the  positive multiples of 18. Then, for any $N\geq 3$,
$$D_S(N)=2D_T(N-3).$$}

\begin{conjecture}\label{c11}
Condition (i) of Theorem \ref{main} holds for $N_0= 2$,  $C=10$, $m=2$, and
$$(A_1,\dots,A_{12})=(1,1,1,2,2,3,3,3,4,4,5,5), (B_1,\dots,B_{12})=(0,0,1,1,2,2,2,3,3,4,4,4).$$
\end{conjecture}

\noindent
\textbf{Corollary  to Conjecture \ref{c11}.}
\emph{Let $S$ be the set containing 2 copies of the  positive integers that are not multiples of 10, one more copy of the odd positive integers, and one more copy of the odd positive multiples of 5; let $T$ be the set containing  2 copies of the  positive integers that are not odd multiples of 5, one more copy of the even positive integers, and one more copy of the  positive multiples of 10. Then, for any $N\geq 2$,
$$D_S(N)=2D_T(N-2).$$}

\begin{conjecture}\label{c12}
Condition (i) of Theorem \ref{main} holds for $N_0= 1$,  $C=10$, $m=1$, and
$$(A_1,\dots,A_{12})=(0,0,1,2,2,2,3,4,4,4,5,5), (B_1,\dots,B_{12})=(0,0,1,1,1,2,3,3,3,4,5,5).$$
\end{conjecture}

\noindent
\textbf{Corollary  to Conjecture \ref{c12}.}
\emph{Let $S$ be the set containing 3 copies of the even positive integers, one  copy of the odd positive integers, 3 more copies of the odd positive multiples of 5, and one more copy of the  positive multiples of 10; let $T$ be the set containing  3 copies of the odd positive integers, one  copy of the even positive integers, one more copy of the odd positive multiples of 5, and 3 more copies of the  positive multiples of 10. Then, for any $N\geq 1$,
$$D_S(N)=D_T(N-1).$$}

\begin{conjecture}\label{c14}
Condition (i) of Theorem \ref{main} holds for $N_0= 2$,  $C=12$, $m=2$, and
$$(A_1,\dots,A_{12})=(0,2,2,2,2,3,3,3,3,4,4,6), (B_1,\dots,B_{12})=(0,0,0,1,2,2,3,3,4,4,5,6).$$
\end{conjecture}

\noindent
\textbf{Corollary  to Conjecture \ref{c14}.}
\emph{Let $S$ be the set containing 2 copies of the even positive integers,  2 more copies of the positive integers congruent to $\pm 2$ modulo 12, and 4 copies of the odd multiples of 3; let $T$ be the set containing 2 copies of the even positive integers, 4 more copies of the positive multiples of 12, one copy of the odd positive integers, and one more copy of the  odd multiples of 3. Then, for any $N\geq 2$,
$$D_S(N)=4D_T(N-2).$$}

\begin{conjecture}\label{c15}
Condition (i) of Theorem \ref{main} holds for $N_0= 2$,  $C=12$, $m=2$, and
$$(A_1,\dots,A_{12})=(1,1,2,2,3,3,3,3,5,5,6,6), (B_1,\dots,B_{12})=(0,0,1,1,3,3,3,3,4,4,5,5).$$
\end{conjecture}

\noindent
\textbf{Corollary  to Conjecture \ref{c15}.}
\emph{Let $S$ be the set containing 2 copies of the positive integers that are not multiples of 4, and 2 more copies of the positive multiples of 3 that are not multiples of 4; let $T$ be the set containing 2 copies of the positive integers that are not congruent to 2 modulo 4, and 2 more copies of the positive multiples of 3 that are not congruent to 2 modulo 4. Then, for any $N\geq 2$,
$$D_S(N)=2D_T(N-2).$$}

\begin{conjecture}\label{c16}
Condition (i) of Theorem \ref{main} holds for $N_0= 2$,  $C=12$, $m=2$, and
$$(A_1,\dots,A_{12})=(1,1,1,2,3,3,4,4,5,5,5,6), (B_1,\dots,B_{12})=(0,1,1,1,2,2,3,3,4,5,5,5).$$
\end{conjecture}

\noindent
\textbf{Corollary  to Conjecture \ref{c16}.}
\emph{Let $S$ be the set containing 2 copies of the positive integers that are not congruent to 0 or $\pm 2$ modulo 12, one  copy of the positive integers that are  congruent to $\pm 2$ modulo 12, and one more copy of the positive integers that are  congruent to $\pm 1$ modulo 6; let $T$ be the set containing  2 copies of the positive integers that are not congruent to 6 or $\pm 4$ modulo 12, one  copy of the positive integers that are  congruent to $\pm 4$ modulo 12, and one more copy of the positive integers that are  congruent to $\pm 1$ modulo 6. Then, for any $N\geq 2$,
$$D_S(N)=D_T(N-2).$$}

\begin{conjecture}\label{c17}
Condition (i) of Theorem \ref{main} holds for $N_0= 1$,  $C=12$, $m=1$, and
$$(A_1,\dots,A_{12})=(0,1,2,2,2,3,3,4,4,4,4,5), (B_1,\dots,B_{12})=(0,1,1,1,2,2,3,3,4,5,5,5).$$
\end{conjecture}

\noindent
\textbf{Corollary  to Conjecture \ref{c17}.}
\emph{Let $S$ be the set containing one copy of the  positive integers that are not odd multiples of 6, one more copy of the positive multiples of 3 that are not odd multiples of 6, 2 more copies of the positive integers that are  congruent to $\pm 2$ modulo 12, and 3 more copies of the positive integers that are  congruent to $\pm 4$ modulo 12; let $T$ be the set containing   2 copies of the positive integers that are not congruent to 6 or $\pm 4$ modulo 12, one  copy of the positive integers that are  congruent to $\pm 4$ modulo 12, and one more copy of the positive integers that are  congruent to $\pm 1$ modulo 6. Then, for any $N\geq 1$,
$$D_S(N)=D_T(N-1).$$}

\begin{conjecture}\label{c17'}
Condition (i) of Theorem \ref{main} holds for $N_0= 2$,  $C=12$, $m= 1$, and
$$(A_1,\dots,A_{12})=(1,1,1,2,3,3,4,4,5,5,5,6), (B_1,\dots,B_{12})=(0,1,2,2,2,3,3,4,4,4,4,5).$$
\end{conjecture}

\noindent
\textbf{Corollary  to Conjecture \ref{c17'}.}
\emph{Let $S$ be the set containing 2 copies of the positive integers that are not congruent to 0 or $\pm 2$ modulo 12, one  copy of the positive integers that are  congruent to $\pm 2$ modulo 12, and one more copy of the positive integers that are  congruent to $\pm 1$ modulo 6; let $T$ be the set containing  one copy of the  positive integers that are not odd multiples of 6, one more copy of the positive multiples of 3 that are not odd multiples of 6, 2 more copies of the positive integers that are  congruent to $\pm 2$ modulo 12, and 3 more copies of the positive integers that are  congruent to $\pm 4$ modulo 12. Then, for any $N\geq 2$,
$$D_S(N)=D_T(N-1).$$}

\begin{conjecture}\label{c18}
Condition (i) of Theorem \ref{main} holds for $N_0= 1$,  $C=12$, $m=1$, and
$$(A_1,\dots,A_{12})=(0,1,1,2,2,2,4,4,5,5,6,6), (B_1,\dots,B_{12})=(0,0,1,1,2,2,4,4,4,5,5,6).$$
\end{conjecture}

\noindent
\textbf{Corollary  to Conjecture \ref{c18}.}
\emph{Let $S$ be the set containing 2 copies of the positive integers that are not odd multiples of 3, one more copy of the positive integers that are  congruent to $\pm 2$ modulo 12, and 2 more copies of the  positive odd multiples of 6; let $T$ be the set containing  2 copies of the positive integers that are not odd multiples of 3, one more copy of the positive integers that are  congruent to $\pm 4$ modulo 12, and 2 more copies of the  positive multiples of 12. Then, for any $N\geq 1$,
$$D_S(N)=2D_T(N-1).$$}

\begin{conjecture}\label{c19}
Condition (i) of Theorem \ref{main} holds for $N_0= 1$,  $C=12$, $m=0$, and
$$(A_1,\dots,A_{12})=(0,1,1,1,3,3,4,4,4,5,5,5), (B_1,\dots,B_{12})=(1,1,1,2,2,2,3,3,5,5,5,6).$$
\end{conjecture}

\noindent
\textbf{Corollary  to Conjecture \ref{c19}.}
\emph{Let $S$ be the set containing  2 copies of the positive integers that are not congruent to 2 modulo 4, one more copy of the positive integers that are  congruent to $\pm 1$ modulo 6, and one more copy of the positive integers that are  congruent to $\pm 4$ modulo 12; let $T$ be the set containing 2 copies of the positive integers that are not multiples of 4, one more copy of the positive integers that are  congruent to $\pm 1$ modulo 6, and one more copy of the positive integers that are  congruent to $\pm 2$ modulo 12. Then, for any $N\geq 1$,
$$D_S(N)=D_T(N).$$}

\begin{conjecture}\label{c21}
Condition (i) of Theorem \ref{main} holds for $N_0= 1$,  $C=12$, $m=1$, and
$$(A_1,\dots,A_{12})=(0,0,1,2,3,3,4,4,4,5,6,6), (B_1,\dots,B_{12})=(0,0,1,2,2,2,3,3,4,5,6,6).$$
\end{conjecture}

\noindent
\textbf{Corollary  to Conjecture \ref{c21}.}
\emph{Let $S$ be the set containing  4 copies of the positive multiples of 6, one copy of the positive integers that are  congruent to $\pm 1$ modulo 6,  one  copy of the positive integers that are  congruent to $\pm 2$ modulo 6, 2 more copies  of the positive integers that are  congruent to $\pm 4$ modulo 12, and 2  copies of the odd positive multiples of 3; let $T$ be the set containing  4 copies of the positive multiples of 6, one copy of the positive integers that are  congruent to $\pm 1$ modulo 6,  one  copy of the positive integers that are  congruent to $\pm 4$ modulo 12, 2  copies of the odd positive multiples of 3, and 3 copies  of the positive integers that are  congruent to $\pm 2$ modulo 12. Then, for any $N\geq 1$,
$$D_S(N)=D_T(N-1).$$}

\begin{conjecture}\label{c22}
Condition (i) of Theorem \ref{main} holds for $N_0= 1$,  $C=12$, $m=1$, and
$$(A_1,\dots,A_{12})=(0,1,1,2,3,3,3,3,4,5,5,6), (B_1,\dots,B_{12})=(0,0,1,2,2,2,3,3,4,5,6,6).$$
\end{conjecture}

\noindent
\textbf{Corollary  to Conjecture \ref{c22}.}
\emph{Let $S$ be the set containing 2 copies of the positive multiples of 6, 2 copies of the positive integers that are  congruent to $\pm 1$ modulo 6,  one  copy of the positive integers that are  congruent to $\pm 2$ modulo 6, and 4 copies of the odd positive multiples of 3; let $T$ be the set containing 4 copies of the positive multiples of 6, one copy of the positive integers that are  congruent to $\pm 1$ modulo 6,  one  copy of the positive integers that are  congruent to $\pm 2$ modulo 6, 2 more copies of the positive integers that are  congruent to $\pm 2$ modulo 12, and 2 copies of the odd positive multiples of 3. Then, for any $N\geq 1$,
$$D_S(N)=2D_T(N-1).$$}

\begin{conjecture}\label{c22'}
Condition (i) of Theorem \ref{main} holds for $N_0= 1$,  $C=12$, $m= 0$, and
$$(A_1,\dots,A_{12})=(0,0,1,2,3,3,4,4,4,5,6,6), (B_1,\dots,B_{12})=(0,1,1,2,3,3,3,3,4,5,5,6).$$
\end{conjecture}

\noindent
\textbf{Corollary  to Conjecture \ref{c22'}.}
\emph{Let $S$ be the set containing  4 copies of the positive multiples of 6, one copy of the positive integers that are  congruent to $\pm 1$ modulo 6,  one  copy of the positive integers that are  congruent to $\pm 2$ modulo 6, 2 more copies  of the positive integers that are  congruent to $\pm 4$ modulo 12, and 2  copies of the odd positive multiples of 3; let $T$ be the set containing 2 copies of the positive multiples of 6, 2 copies of the positive integers that are  congruent to $\pm 1$ modulo 6,  one  copy of the positive integers that are  congruent to $\pm 2$ modulo 6, and 4 copies of the odd positive multiples of 3. Then, for any $N\geq 1$,
$$D_S(N)=\frac{1}{2}D_T(N).$$}

\begin{conjecture}\label{c24}
Condition (i) of Theorem \ref{main} holds for $N_0= 1$,  $C=14$, $m=1$, and
$$(A_1,\dots,A_{12})=(0,1,1,2,3,3,4,5,5,6,7,7), (B_1,\dots,B_{12})=(0,0,1,2,2,3,4,4,5,6,6,7).$$
\end{conjecture}

\noindent
\textbf{Corollary  to Conjecture \ref{c24}.}
\emph{Let $S$ be the set containing one copy of the even positive integers, 2 copies of the odd positive integers, one more copy of the positive multiples of 14, and 2 more copies of the odd positive multiples of 7; let $T$ be the set containing 2 copies of the even positive integers, one copy of the odd positive integers, 2 more copies of the positive multiples of 14, and one more copy of the odd positive multiples of 7. Then, for any $N\geq 1$,
$$D_S(N)=2D_T(N-1).$$}

\begin{conjecture}\label{c25}
Condition (i) of Theorem \ref{main} holds for $N_0= 2$,  $C=16$, $m=2$, and
$$(A_1,\dots,A_{12})=(1,1,2,3,3,4,5,5,6,7,7,8), (B_1,\dots,B_{12})=(0,1,1,2,3,3,4,5,5,6,7,7).$$
\end{conjecture}

\noindent
\textbf{Corollary  to Conjecture \ref{c25}.}
\emph{Let $S$ be the set containing 2 copies of the odd positive integers, one copy of the even positive integers that are not multiples of 16, and one more copy of the  positive odd multiples of 8; let $T$ be the set containing  2 copies of the odd positive integers, one copy of the even positive integers that are not odd multiples of 8, and one more copy of the  positive  multiples of 16. Then, for any $N\geq 2$,
$$D_S(N)=D_T(N-2).$$}

\begin{conjecture}\label{c0}
Condition (i) of Theorem \ref{main} holds for $N_0= 3$,  $C=30$, $m=3$, and
\small
$$(A_1,\dots,A_{12})=(1,3,3,5,5,7,9,9,11,13,15,15), (B_1,\dots,B_{12})=(0,0,2,4,6,6,8,10,10,12,12,14).$$
\normalsize
\end{conjecture}

\noindent
\textbf{Corollary  to Conjecture \ref{c0}.}
\emph{Let $S$ be the set containing one copy of the odd positive integers, one more copy of the odd positive multiples of 3, one more of the odd positive multiples of 5,  and one more of the odd positive multiples of 15; let $T$ be the set containing one copy of the even positive integers, one more copy of the positive multiples of 6, one more of the  positive multiples of 10,  and one more of the  positive multiples of 30. Then, for any $N\geq 3$,
$$D_S(N)=2D_T(N-3).$$}

\begin{remark}
Notice that the last conjecture is already known to be true analytically, thanks to a result of N. D. Baruah and B. C. Berndt (see \cite{BB}, Theorem 8.1). In fact, as the authors of \cite{BB} remarked in the introduction to their paper, the partition identity  of Corollary  to Conjecture \ref{c0} is particularly interesting, since it  arises from another exceptional modular equation discovered by Ramanujan. However, unlike the five equations of the Schr\"oter, Russell and Ramanujan type, the degree of the modular equation corresponding to Corollary  to Conjecture \ref{c0} is 15, hence not a prime.
\end{remark}

\section*{Acknowledgements} This work, along with the previous paper \cite{CSFZ1}, is the result of the first author's MIT senior  thesis, done in Summer and Fall 2011 under the supervision of the second author, and funded by the Institute through two UROP grants. We wish to thank Nayandeep Deka Baruah  for pointing us to Theorem 8.1 of reference \cite{BB}, and the anonymous referee for a careful reading of our manuscript and  helpful comments. The second author warmly thanks Richard Stanley for his terrific hospitality  during calendar year 2011, the MIT Math Department for partial financial support, and Dr. Gockenbach and the Michigan Tech Math Department, from which he was on partial leave, for extra Summer support.


\end{document}